\documentclass[11pt]{article}

 \usepackage[utf8]{inputenc}
 \usepackage[T1]{fontenc}  
\usepackage[english]{babel}
\usepackage{amsfonts}

\usepackage{amsmath,amsxtra,amsthm,latexsym,mathabx}
\usepackage{amssymb} 
\usepackage{accents}

\usepackage{hyperref}

\newtheorem{thm}{Theorem}[section]
 \newtheorem{cor}[thm]{Corollary}
 \newtheorem{lem}[thm]{Lemma}
 \newtheorem{prop}[thm]{Proposition}

 \theoremstyle{definition}
 \newtheorem{defn}{Definition}[section]
 \theoremstyle{remark}
 \newtheorem{rem}{Remark}[section]

 \newtheorem{ex}[thm]{Example}
 
 \numberwithin{equation}{section}

\renewcommand{\>}{\rangle}

\newcommand{\al}{\alpha}
\newcommand{\Th}{\Theta}

\newcommand{\la}{\lambda}
\newcommand{\de}{\delta}
\newcommand{\vep}{\varepsilon}

\newcommand{\ga}{\gamma}
\newcommand{\Ga}{\Gamma}
\newcommand{\ka}{\kappa}
\newcommand{\Om}{\Omega}

\newcommand{\si}{\sigma}
\newcommand{\De}{\Delta}

\newcommand{\om}{\omega}

\newcommand{\BB}{\mathbb{B}} \newcommand{\CC}{\mathbb{C}} 
 \newcommand{\EE}{\mathbb{E}} 
 \newcommand{\GG}{\mathbb{G}} 
\newcommand{\HH}{\mathbb{H}} \newcommand{\II}{\mathbb{I}} 
 
\newcommand{\LL}{\mathbb{L}}  
\newcommand{\NN}{\mathbb{N}} 
\newcommand{\PP}{\mathbb{P}}
\newcommand{\QQ}{\mathbb{Q}}
\newcommand{\RR}{\mathbb{R}}
\newcommand{\TT}{\mathbb{T}}

\usepackage[bbgreekl]{mathbbol}
\DeclareSymbolFontAlphabet{\mathbbm}{bbold}
\DeclareSymbolFontAlphabet{\mathbb}{AMSb}%
\DeclareMathAlphabet{\mathmybb}{U}{bbold}{m}{n}

\newcommand{\ab}{\alpha}

\newcommand{\one}{\mathbbm{1}}

\newcommand{\Ac}{\mathcal{A}}
\newcommand{\Bc}{\mathcal{B}}
\newcommand{\Cc}{\mathcal{C}}

\newcommand{\Ec}{\mathcal{E}}
\newcommand{\Fc}{\mathcal{F}}

\newcommand{\Ic}{\mathcal{I}}
\newcommand{\Kc}{\mathcal{K}}
\newcommand{\Lc}{\mathcal{L}}
 \newcommand{\Nc}{\mathcal{N}} \newcommand{\Oc}{\mathcal{O}}
  
\newcommand{\Tc}{\mathcal{T}}
\newcommand{\Uc}{\mathcal{U}}

\newcommand{\Zc}{\mathcal{Z}}

  \newcommand{\Ff}{\mathfrak{F}}
 \newcommand{\Hf}{\mathfrak{H}}

 \newcommand{\Sf}{\mathfrak{S}} \newcommand{\Tf}{\mathfrak{T}} 
\newcommand{\Xf}{\mathfrak{X}} \newcommand{\Yf}{\mathfrak{Y}}
\newcommand{\af}{\mathfrak{a}} \newcommand{\bfr}{\mathfrak{b}}

\newcommand{\Dr}{\mathrm{D}}
\newcommand{\Nr}{\mathrm{N}}

\newcommand{\pr}{\mathrm{p}}  
\newcommand{\rr}{\mathrm{r}}

\newcommand{\ubf}{\mathbf{u}}
\newcommand{\vbf}{\mathbf{v}}

\newcommand{\wt}{\widetilde}
\newcommand{\wh}{\widehat}

\newcommand{\<}{\langle}
\newcommand{\uph}{\upharpoonright}
\newcommand{\imb}{\hookrightarrow}

\newcommand{\ii}{\mathrm{i}}
\newcommand{\ee}{\mathrm{e}}
\newcommand{\dd}{\mathrm{d}}
\newcommand{\pa}{\partial}

\newcommand{\fin}{\mathrm{fin}}
\newcommand{\comp}{\mathrm{comp}}

\newcommand{\disc}{\mathrm{disc}}
\newcommand{\ess}{\mathrm{ess}}
\newcommand{\sym}{\mathrm{sym}}

\DeclareMathOperator{\im}{Im}
\DeclareMathOperator{\re}{Re}
\DeclareMathOperator{\dom}{dom}

\DeclareMathOperator{\Span}{span}
\DeclareMathOperator{\Gr}{Gr}
\DeclareMathOperator{\Ncone}{\mathfrak{nc}}

\DeclareMathOperator{\Div}{div}
\DeclareMathOperator{\gradm}{\mathbf{grad}}

\DeclareMathOperator{\gradn}{\mathbf{grad}_0}
\DeclareMathOperator{\Divn}{\Div_0}

\DeclareMathOperator{\Tr}{{Tr}}
\DeclareMathOperator{\Hom}{{\mathcal{LH}}}

\newcommand{\Ao}{A}
\newcommand{\Bo}{C}
\newcommand{\A}{\mathcal{A}}

\newcommand{\M}{M}

\newcommand{\gan}{\ga_{\mathrm{n}}}

\newcommand{\x}{\mathbf{x}}
\newcommand{\n}{\mathbf{n}}
\newcommand{\Hs}{\Hf}

\newcommand{\Po}{P}
\newcommand{\So}{S}
\newcommand{\D}{\Oc}
\newcommand{\paD}{{\pa \D}}

\newcommand{\DtN}{{\mathrm{DtN}}}
\newcommand{\NtD}{{\mathrm{NtD}}}
\newcommand{\diag}{{\mathrm{diag}}}
\newcommand{\cross}{{\natural}}

\begin{document}
\title{Random acoustic boundary conditions and Weyl's law for Laplace-Beltrami operators on non-smooth boundaries} 
\author{}
\date{}
\maketitle
  
{ \center{\large Illya M. Karabash
\\[2ex] 
}   }  
  
 {\small 
\center{$^{\text{a}}$ Institute for Applied Mathematics, the University of Bonn, Bonn, Germany\\}

\center{E-mails: karabash@iam.uni-bonn.de, i.m.karabash@gmail.com\\} 
}

\medskip

\vspace{4ex}

\begin{abstract}
Motivated by engineering and photonics research on  resonators in random or uncertain environments, we study rigorous randomizations of boundary conditions for  wave equations of the acoustic-type in Lipschitz domains $\D$.
First, a parametrization of essentially all m-dissipative boundary condition by contraction operators in the boundary $L^2$-space is constructed with the use of m-boundary tuples (boundary value spaces). We consider randomizations of these contraction operators that  lead to acoustic operators random in the resolvent sense. To this end, 
the use of Neumann-to-Dirichlet maps and Krein-type resolvent formulae is crucial. We give a description of  random m-dissipative boundary conditions that produce acoustic operators with almost surely (a.s.) compact resolvents, and so, also with a.s. discrete spectra. 
For each particular applied model, one can choose a specific boundary condition from the constructed class either by means of optimization, or on the base of empirical observations.
A mathematically convenient randomization is constructed in terms of eigenfunctions of the Laplace-Beltrami operator $\De^{\pa \D}$ on the boundary $\pa \D$ of the domain.
We show that for this randomization the compactness of the resolvent is connected with the  Weyl-type asymptotics for the eigenvalues of $\De^{\pa \D}$. 
\end{abstract}

\vspace{1ex}
{\small
\noindent
MSC2020-classes: 
60H25 
35F45   
47B80   
35P05   
58J90 
47B44 
\\[0.5ex]
}

\medskip

{\noindent\small \textsc{Acknowledgement.} 
The author is supported by the Heisenberg Programme (project 509859995) of the German Science Foundation (DFG, Deutsche Forschungsgemeinschaft), by the Hausdorff Center for Mathematics funded by the German Science Foundation under Germany's Excellence Strategy -- EXC-2047/1 -- 390685813, 
and by the CRC 1060 `The mathematics of emergent effects' funded by the German Science Foundation at the University of
Bonn. 
}

\section{\label{s:i}Introduction}

The aim of this paper is a rigorous randomization of dissipative boundary conditions for open resonators (or cavities) modelled by wave equations in bounded domains $\D \subset \RR^d$. 
The domain $\D$, which represents the interior part of an acoustic or optical cavity, is assumed to have a  boundary $\pa \D$ of the Lipschitz regularity. 
Our motivation stems from applied physics and mathematical studies of resonators in a random  environment \cite{DMTSH14,FG20,LCvM22,FAKTN23} and the high-Q design problem for photonic crystals
\cite{LJ13,KLV17,KKV20,EK24} in the case of a coupling with uncertain external structures  \cite{LJ13,KLV17,EK22}.

Random boundary conditions in this paper are constructed  in the context of linear wave equations of the acoustic-type 
$\beta (\x) \pa_t^2  p (\x,t)  = \nabla_\x \cdot \left[ (\ab (\x))^{-1} \nabla_\x  p (\x,t) \right]$, \quad $t>0$, $ \x \in \D\subset \RR^d$,
corresponding to a deterministic, but possibly non-homogeneous and anisotropic, medium inside the resonator. The internal medium is  described by deterministic 
uniformly positive material parameters $\beta \in L^\infty (\D,\RR)$ and $\ab \in L^\infty (\D,\RR_\sym^{d \times d})$, see Section \ref{s:DisBC}. 
If $d=2$, the acoustic  equation 
is essentially equivalent to one of dimensionally reduced versions of Maxwell systems, and is used often as a  model for propagation of electromagnetic waves  \cite{FK96,ACL18}.

We restrict our attention to linear time-independent boundary conditions, but we do not assume their locality.
While deterministic local boundary conditions of Leontovich (or impedance) type are  a standard  choice for absorbing (or dissipative) boundary conditions \cite{LL84,CK13,ACL18,YI18,EK22},
the exact Dirichlet-to-Neumann (DtN) radiation boundary conditions on spherical boundaries can be expressed in the case of idealized homogeneous outer medium via spherical harmonics \cite{KG89} and are inherently  non-local (for radiation boundary conditions connected with DtN maps, see \cite{FJL21,FG20,SST24}). 
The papers of Kurula \& Zwart \cite{KZ15} and Skrepek \cite{S21}
studied wide classes of general deterministic, not necessarily local,  m-dissipative boundary conditions in the acoustic-type settings close to those of the present paper. For Maxwell systems, boundary conditions with possibly non-local deterministic impedance operators were discussed in the monograph \cite{ACL18} and studied in detail by Eller \& the author in \cite{EK22,EK24}.

In the deterministic case, boundary conditions and their coefficients are designed to take into account the interaction of interior oscillations with the structure of the outer medium near the cavity \cite{ACL18,YI18}. 
 In the case of an uncertain or stochastic ambient medium, it makes sense to perform randomization in the corresponding classes of deterministic m-dissipative  boundary conditions.

Using m-boundary tuples \cite{EK22}, we introduce in Section \ref{s:DisBC} a special parametrization of essentially all  m-dissipative boundary conditions by operators in the boundary space $L^2 (\pa \D)$, and then randomize this parametrization in such a way that the resulting stochastic generator of the acoustic semigroup has reasonable probabilistic and spectral properties. 
This gives a wide class of random boundary conditions. For particular engineering settings or for a particular model of the surrounding stochastic environment, a certain random boundary condition can be chosen from this class either by means of mathematical optimization or collecting empirical evidence. Below we formulate and justify assumptions (A1)-(A3) that we consider to be reasonable criteria for the choice of random boundary conditions.

In order to address the loss of energy through the boundary, it is convenient  if the energy
of the system is expressed via the norm of a Hilbert space where the dynamics is considered. This point of view corresponds to the 1st order acoustic system, which  we write following \cite{L13} in the Schrödinger style $ \ii \pa_t \Psi   = \af_{\ab,\beta} \Psi $ using  the 1st order spatial differential operation 
\begin{gather} \label{e:Ac}
\quad 
\text{ } \af_{\ab,\beta} : \begin{pmatrix} \vbf  \\ p \end{pmatrix} \mapsto \frac{1}{\ii} \begin{pmatrix}  \ab^{-1} \nabla p \\
 \beta^{-1} \nabla \cdot \vbf \end{pmatrix}.
\end{gather} 
The dynamics of $\Psi = \{\vbf,p\}$ is considered in the `energy space' $\LL^2_{\ab,\beta} (\D) $ that   
coincides with the orthogonal sum $ L^2 (\D,\CC^d) \oplus L^2 (\D)$ as a linear space, but
has another (equivalent) norm $ \| \cdot \|_{\LL^2_{\ab,\beta} }$ corresponding to the energy inside the resonator 
\[
 \| \{\vbf,p\}  \|_{\LL^2_{\ab,\beta} }^2  
 = ( \ab \vbf | \vbf)_{L^2 (\D,\CC^d) }  + ( \beta p  | p )_{L^2 (\D) } = \int_\D ( \ab  \vbf \vert \vbf )_{\CC^d}  + \int_\D \beta |p|^2 . 
\]

A general time-independent deterministic linear boundary condition can be written as 
$  \Cc^0 \ga_0 (p) +  \Cc^1 \ga_n (\vbf) = 0 $. The operator coefficients $\Cc^0$ and $\Cc^1$  are combined with the bounded operators of the scalar trace $\ga_0: H^1 ( \D) \to H^{1/2} (\pa \D)$ and  the normal trace $\gan : \HH (\Div, \D) \to H^{-1/2} (\pa \D)$ (see \cite{M03,CK13,ACL18}), where 
\begin{gather*} 
 \HH (\Div, \D) := \{ \ubf \in L^2 (\D,\CC^d) \ : \nabla \cdot \ubf \in L^2 (\D)\}.  
 \end{gather*}
This space and other spaces built as domains (of definition) of differential operators are equipped with the corresponding graph norms.
The assumption of dissipativity of a deterministic boundary condition can  be written rigorously in the following form:
\begin{itemize}
\item[(A1)] The differential operation $\af_{\ab,\beta} $
equipped with a boundary condition  should produce a certain m-dissipative operator $\wh \A$ in  $\LL^2_{\ab,\beta} (\D) $ (see Definition \ref{d:bc} for details).
 \end{itemize}

We use the convention of \cite{E12,EK22} concerning dissipative and m-dissipative operators (see Section \ref{s:DisBC}). 
Assumption (A1) ensures that the energy $\frac12 \| \Psi (t) \|^2_{\LL^2_{\ab,\beta}} = \frac12 \|\ee^{-\ii t \wh \A} \Psi (0) \|_{\LL^2_{\ab,\beta}}^2 $ is non-increasing in time, i.e., $(-\ii) \wh \A$ is a generator of a contraction semigroup $\ee^{-\ii t \wh \A}$, see \cite{P59,Kato,E12}.

In the case of infinite-dimensional spaces, it is difficult to put randomized operators  into standard settings of stochastic measurability of random variables \cite{DZ92}.
In order to address the measurability with respect to (w.r.t.) the underlying probability space $(\Om, \Fc,\PP)$, one adapts various specific definitions of random operators  \cite{BR72,S84,PF92}. In this paper, several such definitions are used and distinguished by special names, e.g., random bounded operators and random-m-dissipative operators, see  Section \ref{s:RBC}. 
We implement the requirement of the probabilistic measurability via the following condition:

\begin{itemize}
\item[(A2)]  The differential expression $\af_{\al,\beta}$ equipped with a randomized m-dissipative boundary condition  should define an acoustic operator $\wh \A: \om \mapsto \wh \A_\om$ that is random in the resolvent sense (we call such operators random-m-dissipative, see Definition \ref{d:rmDis}).  
 \end{itemize}


An accurate  description of random spectra in terms of stochastic point processes is possible for random  operators only if they have a purely discrete spectrum with probability 1. However,
a 1st order m-dissipative acoustic operator $\wh \A$ satisfying (A1) cannot have a purely discrete spectrum if $d \ge 2$.
Indeed, the space 
\begin{gather*}  
\text{
$\HH_0 (\Div 0,\D) = \{\ubf \in \HH (\Div,\D) : \nabla \cdot \vbf = 0, \ \ga_n \vbf = 0\}$}
\end{gather*}
is infinite-dimensional \cite{L13}. The symmetric acoustic operator $\A$ considered on the minimal reasonable domain of definition 
$\dom \A = \HH_0 (\Div,\D) \times  H_0^1 (\D) $ has an infinite-dimensional kernel $\ker \wh A  \supseteq \HH_0 (\Div 0,\D) \oplus \{0\} $.
So, the essential spectrum $\si_\ess (\wh \A)$ is non-empty for any m-dissipative extension $\wh \A$ of $\A$. 
However, any such $\wh \A$ admits an orthogonal decomposition 
with respect to  the space $\HH_0 (\Div 0,\D) \oplus \{0\}$ and its orthogonal complement 
\[
\GG_{\ab,\beta} = \LL^2_{\ab,\beta} (\D) \ominus \left( \HH_0 (\Div 0,\D) \oplus \{0\} \right) .
\]
This naturally leads to the study of the discrete spectrum of the part $\wh \A |_{\GG_{\ab,\beta}}$  of $\wh \A$ in the closed subspace $\GG_{\ab,\beta}$. In particular, if the operator $\wh \A |_{\GG_{\ab,\beta}}$
has a compact resolvent, then  it has a purely discrete spectrum $\si (\wh \A |_{\GG_{\ab,\beta}}) = \si_\disc (\wh \A |_{\GG_{\ab,\beta}})$.

One of the main goals of this paper is a construction of random boundary conditions such that the following \emph{partial compactness property} is satisfied: 
\begin{itemize}
\item[(A3)]  A random-m-dissipative acoustic operator $\wh \A$ is supposed to have the property that its part $\wh \A |_{\GG_{\ab,\beta}}$ has a compact resolvent with probability 1.
 \end{itemize}

Generally, for partial differential operators, the resolvent compactness is atypical. In the case of deterministic  boundary conditions, this can be seen from abstract results of the theory of boundary triples for symmetric operators with infinite deficiency indices \cite{GG91}.  

Stochastic settings allows us to see how much atypical the partial compactness property (A3) is. In Section \ref{s:PartComp} we provide general necessary and sufficient conditions for the validity of (A3) and give several corresponding  constructions of specific  random boundary conditions. 

Using a unitary equivalence, the statements concerning the operator $\wh \A |_{\GG_{\ab,\beta}}$ can be equivalently reformulated  in terms of the 2nd order acoustic operators $\wh \Bc$ defined via the differential expression 
$
\bfr_{\ab,\beta} : \begin{pmatrix} u  \\ p   \end{pmatrix} \mapsto \ii\begin{pmatrix} p  \\
\beta^{-1} \nabla \cdot (\ab^{-1} \nabla u)  \end{pmatrix} .
$
The associated acoustic equation takes the form $\ii \pa_t \Phi = \bfr_{\ab,\beta} \Phi$ with
 $\Phi = \{u,p\}$, where $u$ is the momentum potential, and so,  $\vbf$ has to be replaces in the random boundary conditions by $(-\ab^{-1}) \nabla u$.
In these settings, the class of random boundary conditions that we construct in order to fulfil (A2) can be  written as 
\begin{gather*}
(K+I) (I+\De^\paD)^{1/4} \ga_0 (p)  - (K-I) (I+\De^\paD)^{-1/4} \gan (\ab^{-1} \nabla u) = 0 ,
\end{gather*}
where $K$ is a random contraction in $L^2 (\pa \D)$, and $\De^\paD$
is the nonnegative Laplace-Beltrami operator on $\pa \D$ (defined as in \cite{GMMM11}).
Assumption (A3) introduces a restriction on $K$, namely,
the random operator $K+I$ has to be a.s. compact in $L^2 (\pa \D)$ (Theorem \ref{t:AM-disComp}).


We consider also in Theorem \ref{t:Imp}  a subclass of m-dissipative boundary conditions, which can be written in the form $\Zc \ga_0 (p) = \gan (\vbf)$ with the use of 
\[
\text{\emph{impedance operators}
$\Zc:\dom \Zc \subseteq H^{1/2} (\pa \D) \to H^{-1/2} (\pa \D)$. }
\]
A particularly convenient randomization of $\Zc$ is provided by the case where $\Zc$ is diagonal w.r.t.
the basis $\{y_j\}_{j\in \NN}$ of eigenfunctions of the Laplace-Beltrami operator $\De^\paD$. Assumption (A2) is satisfied, in particular, if the  diagonal entries of the corresponding matrix of $\Zc$ are random variables $\zeta_j$, $j \in \NN$,
with values in the closed right half-plane $ \overline{\CC}_\rr := \{z \in \CC : \re z \ge 0\}$.

In the case where the diagonal entries $\zeta_j$ are independent and identically distributed (i.i.d.), we  characterize the impedance operators satisfying (A3) in terms of the random variables $\Nc_{\De^{\pa \D}} (|\zeta_j|^2/\de^2)$,
where \[
\textstyle \Nc_{\De^{\pa \D}} (\la) := \# \{ k \in \NN : \mu_k \le \la\} = \sum_{\mu_k \le \la} 1 , \qquad \la \in \RR,
\]
is the counting function for eigenvalues of  the Laplace-Beltrami operator $\De^\paD$
and $\de>0$ is a small parameter, see Theorem \ref{t:iid}.
Under certain mild additional assumptions on the regularity of the boundary $\pa \D$, Theorem \ref{t:iid}  
implies a simple criterion for the partial compactness property (A3) in terms of the raw moments of $|\zeta_j|$,
see Corollary \ref{c:Weyl}.

The paper is organized as follows. The main results and some of the proof are collected in Section \ref{s:Main}. In particular, in Section \ref{s:NtDKrRes}, we formulate the results on the Neumann-to-Dirichlet (NtD) maps and Krein-type resolvent formulae, which are needed for the verification of the randomness of the acoustic operators. Especially lengthy proofs and the proofs that require substantial use of  boundary tuples,
linear relations, and the theory of extensions are postponed to subsequent Sections \ref{s:BT+proof}-\ref{s:DiscSp}. The discussion section (Section \ref{s:dis}) contains additional remarks about DtN-maps and Weyl-type asymptotics for eigenvalues of $\De^\paD$.

While differential operators with random coefficients and their stochastic spectra were intensively studied from various points of view,
including Anderson localization and stochastic homogenization (see
\cite{PF92,FKl96,K16,FG20,LO21,LOW24} and references therein), the author of the present paper was not able to find in the existing publications a systematic study of random linear boundary conditions. Effective boundary conditions associated with stochastic homogenization models for outer medium were studied analytically in \cite{FG20} in the context of the 1-dimensional wave equation. The quantitative stochastic homogenization papers \cite{LO21,LOW24}   consider 2- and 3-dimensional conductivity equations and propose the replacement of the stochastic medium outside big rectangular boxes  with  Dirichlet data constructed by a specific deterministic algorithm. Randomized Leontovich boundary conditions for Maxwell systems were discussed briefly in \cite{EK22}. Appropriately randomized boundary conditions are expected to lead to stochastic point processes of random eigenvalues. This topic is closely connected with random continuation resonances, which are often modelled by physicists via nonselfadjoint random matrices  \cite{FS15}. Continuation resonances for  bounded regions of stochastic media surrounded by a homogeneous space were studied in different settings in \cite{K16,AK21}.

\textbf{Notation.} 
Let $\Xf$, $\Xf_1$, and $\Xf_2$ be (complex) Hilbert spaces. 
By $\Lc (\Xf_1,\Xf_2)$ we denote the space of bounded (linear) operators $T : \Xf_1 \to \Xf_2$, while 
$\Hom (\Xf_1,\Xf_2)$ is the set of \emph{linear homeomorphisms} from $\Xf_1$ to $\Xf_2$.
Besides, $\Lc (\Xf) := \Lc (\Xf,\Xf)$ and $\Hom (\Xf) := \Hom (\Xf,\Xf)$.
 By $I_\Xf$ (by $0_{\Lc(\Xf)}$) the identity operator (resp., the zero operator) in $\Xf$ is denoted, although  the subscript is dropped  if the space $\Xf$ is clear from the context. In the notation for the resolvent 
$(T-\la)^{-1}=(T-\la I)^{-1}$,   the identity operator $I$ is often skipped.

By $\one$ we denote the constant function equal to $1$. For $1 \le p \le \infty$, $\ell^p (\NN)$ are the standard complex Banach $\ell^p$-spaces of sequences $\{a_j\}_{j \in \NN}$, and $\Sf_p = \Sf_p (\Xf)$  are the Schatten-von-Neumann ideals of compact operators in $\Xf$.
By $\Gr T := \{\{f,Tf\} \ : \ f \in \dom T\}$ we denote the graph of the operator $T$, and consider it as a normed space with the graph norm \cite{Kato,GG91}. We  use the natural identification of $\Gr T$ and the domain (of definition) $\dom T$ of $T$. 

The spectrum and the set of eigenvalues of an operator $T:\dom T \subseteq \Xf \to \Xf$ are denoted by $\si (T)$ and $\si_\pr (T)$, respectively. The notation $\rho (T) = \CC \setminus \si (T)$ stands for the resolvent set.
A symmetric operator $T$ is called nonnegative (and we write $T \ge 0$) if $(Tf|f)_{\Xf} \ge 0$ for all $f \in \dom T$. The notation $T_1 \le T_2$ means $T_2 -T_1 \ge 0$.  A symmetric operator is uniformly positive if $T \ge \de I$ for a certain constant $\de>0$.  For $T \in \Lc (\Xf)$, $\re T = \frac12(T+T^*)$ and $\im T = \frac{1}{2\ii} (T - T^*)$ are the real and imaginary parts of $T$, respectively.

\section{Main results of the paper}
\label{s:Main}

Let $\D$ be a domain in $\RR^d$ with $d \ge 2$, i.e., $\D$ is a non-empty bounded open connected subset of $\RR^d$.  We assume that $\D$ is a Lipschitz domain, i.e.,
its  boundary $\pa \D$ is of the Lipschitz regularity.
Hence the normal to $\partial \D$ outward unit vector-filed $\n \in L^\infty  (\pa \D, \RR^d)$  is well-defined \cite{M03,ACL18}. The $L^p$-spaces on $\pa \D$ are considered w.r.t. the surface measure of $\partial \D$.

By $L^2 (\D,\CC^d)$ we denote the Hilbert space of $\CC^d$-valued vector fields in $\D$ equipped with the inner product 
$(  \ubf | \vbf )_{L^2 (\D,\CC^d)} =  \int_{\D} \ubf \cdot \overline{\vbf}  =  \int_{\D} (  \ubf  | \vbf  )_{\CC^d} $. 
For $s > 0$, we use  the standard complex Hilbertian Sobolev spaces 
$H^s ( \D) = W^{s,2} ( \D, \CC)$ and $H^s_0 ( \D) = W^{s,2}_0 ( \D, \CC)$, as well as the corresponding spaces of distributions $H^{-s} ( \D) = W^{-s,2} ( \D, \CC)$ that are 
dual to  $H^s_0 ( \D) $ w.r.t. the pivot space $L^2 (\D)$.
On the Lipschitz boundary $\pa \D$, we use the space 
 $L^2 (\pa \D) = L^2  (\pa \D, \CC)$ of 
 scalar $L^2$-functions.
 For $s \in (0,1]$,  one can define \cite{M03,ACL18} the standard complex Hilbertian Sobolev spaces of scalar functions $H^s (\pa \D) $  and their dual spaces $H^{-s} (\pa \D) $  w.r.t. the pivot space $H^0 (\pa \D) =L^2 (\pa \D)$.

\subsection{Deterministic m-dissipative boundary conditions for acoustic systems}
\label{s:DisBC}

By  $\RR_{\sym}^{d\times d}$, a Banach space of real symmetric $d\times d$-matrices (with any norm) is denoted. 

Let the material parameters $\ab  = (\ab_{j,k} )_{j,k=1}^d \in L^\infty (\D, \RR_{\sym}^{d\times d})$ and $\beta \in L^\infty (\D,\RR)$ be uniformly positive. The latter means that $\beta (\x) \ge \beta_0 $ and $\ab (\x) \ge \al_0 \II $ with certain positive constants $\beta_0, \al_0>0$ for almost all (a.a.) $\x \in \D$. In the inequality $\ab (\x) \ge \al_0 \II $,   
$\ab (\x)$ and  the identity $d\times d$-matrix $\II$ are identified with selfadjoint operators
in $\CC^d$.
By $\ab^{-1} (\cdot) \in L^\infty (\D, \RR_{\sym}^{d\times d}) $, we denote the pointwise inversion  $ (\ab (\x))^{-1}$, $\x \in \D$, for the matrix-function $\ab (\cdot)$.

The  Hilbert space of vector fields $ \LL^2_\ab = \LL^2_\ab (\D)$ coincides with the Hilbert space 
$ L^2 (\D,\CC^d) $ as a linear space, but is equipped with the weighted norm $ \| \cdot \|_{\LL^2_\ab}$ defined by
$  \| \vbf  \|_{\LL^2_\ab}^2 =  ( \ab \vbf | \vbf)_{L^2 (\D,\CC^d) }$, i.e., 
 $\LL^2_\ab (\D) = (L^2 (\D,\CC^d), \| \cdot \|_{\LL^2_\ab}).$ 
We use the weighted Hilbert space of $\CC$-valued functions 
$L^2_\beta (\D)$ with 
$\| f  \|_{L^2_\beta} :=  (( \beta f | f)_{L^2 (\D) })^{1/2} $.
The Hilbert space $\LL^2_{\ab,\beta} (\D)$ is 
\[
\text{the orthogonal sum } \quad \LL^2_{\ab,\beta} (\D)  = \LL^2_{\ab,\beta} := \LL^2_{\ab} (\D) \oplus L^2_\beta (\D).
 \]

The aim of this subsection is to describe all m-dissipative operators in $ \LL^2_{\ab,\beta} (\D)$ that can be associated with the 1st order acoustic differential expression $\af_{\ab,\beta}$ of \eqref{e:Ac} by means of a choice of a certain  boundary condition.

A (linear) operator $T:\dom T \subseteq \Xf \to \Xf$ in a Hilbert space $\Xf$ is called  \emph{m-dissipative} if $\CC_+ := \{z \in \CC : \im z >0\}$ 
is a subset of its resolvent set $\rho (T)$
and 
$\| (T- z)^{-1}\| \le (\im z)^{-1}$ for all $z \in \CC_+ $ (see, e.g., \cite{E12,EK22}, and note that there is a variety of other  conventions \cite{GG91,DZ92,Kato,KZ15}). The set of m-dissipative operators is a subset in the class of dissipative operators. An  operator $T$ is called \emph{dissipative 
(accretive)} if $\im (Tf|f)_{\Xf} \le 0$ (resp., $\re (Tf|f)_{\Xf} \ge 0$) 
for all $f \in \dom T$.
A dissipative operator in $\Xf$ is \emph{maximal dissipative} if it is not a proper restriction of another dissipative operator in $\Xf$. 
According to Phillips' criteria of m-dissipativity \cite{P59},
\begin{multline}\label{e:CrM-dis1}
\text{an operator $T$ is m-dissipative if and only if $T$ is closed and maximal dissipative} \\
\text{(or, equivalently, if and only if $T$ is densely defined and maximal dissipative).}
\end{multline}
Lossy systems are described in these settings via the following criterion:
$T$ is m-dissipative if and only if $(-\ii) T$ is a generator of a contraction semigroup  \cite{P59,Kato,E12}.
Note that in our settings a spectrum $\si (T)$ of an m-dissipative operator lies in the closure 
$\overline{\CC}_- $ of the lower complex half-plane $\CC_- := \{\la \in \CC : \im \la <0\}$.

The gradient operator $\gradm : f \mapsto \nabla f $ is considered as an operator from $L^2 (\D)$ to $L^2 (\D,\CC^d)$ with the maximal (natural) domain $H^1 (\D)$.
The divergence operator 
$
\Div : \ubf \mapsto \nabla \cdot \ubf$ 
from $L^2 (\D,\CC^d)$ to $L^2 (\D)$  is considered on its maximal domain  
\[
 \HH (\Div, \D) := \{ \ubf \in L^2 (\D,\CC^d)  : \nabla \cdot \ubf \in L^2 (\D)\}.
\] 
The normal trace    
$\gan : \vbf  \mapsto \n \cdot \vbf (x) \uph_{\pa \D}$, which is understood as a continuous operator $\ga_n \in \Lc \left(\HH (\Div, \D), H^{-1/2} (\pa \D)\right)$,  is surjective.   The scalar trace $\ga_0 \in \Lc (H^1 (\D), H^{1/2} (\pa \D))$ defined by $\ga_0 : p \mapsto p\!\uph_{ \pa \D}$ is also a surjective operator (see  \cite{M03,ACL18}).
The operator $\gradn:= -\Div^*$ is a restriction of $\gradm$ to $\dom \gradn = H^1_0 (\D) $ and has the adjoint $\gradn^*= -\Div$.
So,  $\HH (\Div, \D)$ is a Hilbert space (with the graph norm of the operator $\Div$). 
The closed operator $\Div_0 : \HH_0 (\Div,\D) \subset \LL^2 (\D)   \to L^2 (\D)$ 
defined by $\Div_0 := \gradm^*$
is a restriction of the operator $\Div$ to a narrower domain 
 $\HH_0 (\Div,\D) = \{\ubf \in \HH (\Div, \D) \ : \ \gan (\ubf) = 0\}$, which is a closed subspace of $\HH (\Div, \D) $ (see  \cite{M03}).

With the 1st order acoustic differential expression $\af_{\ab,\beta}$ (see  \eqref{e:Ac}), one can associate  in $ \LL^2_{\ab, \beta} (\D)$  the following closed symmetric operator $\A$ \cite{L13}:
\begin{gather}\label{e:A}
\A \Psi = \A \begin{pmatrix} \vbf \\ p \end{pmatrix} = \frac1\ii \begin{pmatrix} 0 &  \ab^{-1} \gradn \\
 \beta^{-1} \Divn & 0 \end{pmatrix} \begin{pmatrix} \vbf \\ p \end{pmatrix} .
\end{gather}
Here $\Psi = \{  \vbf , p \} \in \dom \A = \HH_0 (\Div,\D) \times  H_0^1 (\D) $. The adjoint operator $\A^*$ is defined by the same differential expression $\af_{\ab,\beta}$, but has a wider domain, $\dom \A^* = \HH (\Div,\D) \times H^1 (\D)$.

\begin{defn} \label{d:bc}
Let operators 
$\Ec_j:\dom \Ec_j \subseteq \dom \A^* \to \Yf$, $j=0,1$, have a certain linear space $\Yf$ as their target space.
An acoustic operator $\wh \A$ associated with an 'abstract condition' $  \Ec_0 \Psi +  \Ec_1 \Psi = 0 $ is defined as the restriction  $\wh \A = \A^* \uph_{\dom \wh A}$
of the operator $\A^*$ to 
\[
\dom \wh \A = \{ \Psi \in \dom \A^* \cap \dom \Ec_0 \cap \dom \Ec_1 \ : \ \Ec_0 \Psi +  \Ec_1 \Psi = 0 \} .
\]
We shall say that $  \Ec_0 \Psi +  \Ec_1 \Psi = 0 $ is an \emph{m-dissipative boundary condition}  if $\dom \A \subseteq \dom \wh \A$ and the operator $\wh \A$ is an m-dissipative operator in $\LL^2_{\ab,\beta} (\D)$.
\end{defn}

For $T \in \Lc \left(L^2 (\pa \D), H^{\pm1/2} (\pa \D) \right)$, we denote by $T^\cross \in \Lc \left( H^{\mp 1/2} (\pa \D), L^2 (\pa \D) \right)$ the adjoint operator w.r.t the (sesquilinear) $L^2 (\pa \D)$-pairing 
$\<\cdot,\cdot\>_{L^2 (\pa \D)}$  of $H^{\pm 1/2} (\pa \D)$ with $ H^{\mp 1/2} (\pa \D)$.
In this case, $(T^\cross)^\cross = T$ 
(for definitions concerning the $L^2 (\pa \D)$-pairing  and $\cross$-adjoint operators, see \cite{EK22} and, in the present paper,  Section \ref{s:MBT}). 

By $\Hom (\Xf_1,\Xf_2)$, we denote the set of linear homeomorphisms from $\Xf_1$ to $\Xf_2$. 
If $T \in \Hom \left(L^2 (\pa \D), H^{1/2} (\pa \D) \right) $, then $T^\cross \in \Hom \left( H^{- 1/2} (\pa \D), L^2 (\pa \D) \right)$ and $(T^{-1})^\cross = (T^\cross)^{-1}$ \cite{EK22}.

Recall that $K \in \Lc (\Xf) $ is a \emph{contraction} (in $\Xf$) if $\|K\| \le 1$, where $\|\cdot\| := \|\cdot\|_{\Lc (\Xf)}$.
The following  theorem essentially describes all m-dissipative boundary conditions for the acoustic system (up to equivalent transformations of boundary conditions that do not change $\dom \wh A$).

\begin{thm} \label{t:AM-dis} 
Let us fix a certain mutually $\cross$-adjoint  pair of linear homeomorphisms
\[
\text{$V \in \Hom \left(L^2 (\pa \D), H^{1/2} (\pa \D) \right)$ and $V^\cross \in \Hom \left( H^{-1/2} (\pa \D), L^2 (\pa \D) \right)$.}
\]
Then following two statements are equivalent:
\begin{itemize}
\item[(i)] An operator $\wh \A$ is an m-dissipative extension of the symmetric acoustic operator $\A$. 
\item[(ii)] There exists a contraction $K$ in $L^2 (\pa \D)$ such that $\wh \A $ is the acoustic operator associated (in the sense of Definition \ref{d:bc}) with the boundary condition 
 \begin{gather} \label{e:mdisBC}
(K+I_{L^2 (\pa \D)}) V^{-1} \ga_0 (p)  + (K-I_{L^2 (\pa \D)}) V^\cross \gan (\vbf) = 0 .
\end{gather}
\end{itemize}
This equivalence  establishes a bijective correspondence between 
contractions $K$ in $L^2 (\pa \D)$ and m-dissipative extensions $\wh \A$ of $\A$. Moreover, $\wh \A = \wh \A^*$ if and only if $K$ is a unitary operator.
\end{thm}

This theorem is proved in Section \ref{s:IntParts} using an m-boundary tuple for $\A^*$.

In order to provide a particular explicit pair of linear homeomorphisms $V^{-1}$ and $V^\cross$ in  
\eqref{e:mdisBC},
we consider on the boundary $\paD$ the associated (nonnegative) Laplace-Beltrami operator 
$
 \De_\paD \in \Lc(H^1 (\pa \D), H^{-1} (\paD) ).
$
We use also another version \[
 \De^\paD: \dom (\De_\paD) \subset L^2 (\pa \D) \to L^2  (\pa \D) \] of the  Laplace-Beltrami  operator, which is a nonnegative selfadjoint operator in $L^2 (\pa \D)$. The operator $\De^\paD$  can be rigorously defined via the Friedrichs representation theorem \cite{Kato}  applied to the nonnegative quadratic form that is equal to the square of the $L^2$-norm of the surface gradient $ \gradm_\paD u $, see \cite{GMMM11}. 
The operator $\De^\paD$ has purely discrete spectrum. The corresponding eigenfunction expansion
 implies that $\De^\paD$
can be considered as a restriction of $ \De_\paD$ to $\dom \De^\paD = \{ u \in H^1 (\pa \D) : \De_\paD u \in L^2 (\paD)\}$.

\begin{lem} \label{l:VV*}
Assume that $V: L^2 (\paD) \to H^{1/2} (\paD)$ is defined by
$V: f \mapsto (I+\De^\paD)^{-1/4} f$.
Then $V \in \Hom \left( L^2 (\paD), H^{1/2} (\paD) \right)$. Its $\cross$-adjoint operator $V^\cross \in \Hom \left(H^{-1/2} (\pa \D), L^2 (\pa \D) \right)$  is a unique extension of 
$ V $ by continuity to  a bounded operator from $H^{-1/2} (\pa \D)$ to $L^2 (\pa \D) $ (in particular, $V^\cross f = V f$ for $f \in L^2 (\paD)$).
\end{lem}

\begin{proof}
The lemma follows from the eigenfunction expansion  \cite{GMMM11} for $\De^\paD$. 
\end{proof}

We denote the set of eigenvalues of an operator $T$ by $\si_\pr (T)$.

Assume  that $1 \not \in \si_\pr (K)$. Then the operator $K-I$ in the boundary condition \eqref{e:mdisBC} has a (possibly unbounded) inverse $(K-I)^{-1}$, and \eqref{e:mdisBC} can be equivalently written in the form 
 \begin{gather} \label{e:Imp}
   \Zc \ga_0 (p) = \gan (\vbf)  
\end{gather}
with $\Zc= (V^\cross)^{-1} (I-K)^{-1} (K+I) V^{-1}$. 
By the analogy with Maxwell systems \cite{ACL18,EK22}, we call the boundary conditions of the form $\Zc \ga_0 (p) = \gan (\vbf)$ \emph{generalized impedance boundary conditions} and call an operator $\Zc:\dom \Zc \subseteq H^{1/2} (\pa \D) \to H^{-1/2} (\pa \D)$ in \eqref{e:Imp} \emph{an impedance operator} (for standard impedance boundary conditions, see  \cite{M03,CK13,ACL18} and references therein).

Following an abstract definition in \cite{EK22}, we call an impedance  operator $\Zc$ \emph{accretive} if $\re \<\Zc f | f \>_{L^2 (\pa \D)} \ge 0$ for all $f \in \dom \Zc$, and \emph{maximal accretive} if it has no proper accretive extensions $\wh \Zc:\dom \wh \Zc \subseteq H^{1/2} (\pa \D) \to H^{-1/2} (\pa \D)$ (recall that an extension $ \wh \Zc$ is \emph{proper} if $\dom \wh \Zc \supsetneqq \dom \Zc$).
In the case where $\Zc$ is maximal accretive and closed, its $\cross$-adjoint operator $Z^\cross:\dom \Zc^\cross \subseteq H^{1/2} (\pa \D) \to H^{-1/2} (\pa \D)$ is also maximal accretive and closed (see \cite{EK22} and also Section \ref{s:Ext} for details).

\begin{thm} \label{t:Imp} 
Let $\wh \A$ be an acoustic operator associated with \eqref{e:Imp} for a certain impedance operator $\Zc:\dom \Zc \subseteq H^{1/2} (\pa \D) \to H^{-1/2} (\pa \D)$. Then $\wh \A$ is m-dissipative if and only if  $\Zc$  is closed and maximal accretive.
Moreover, $\wh \A = \wh \A^*$ if and only if $\Zc^\cross = - \Zc^\cross$. 
\end{thm}

This theorem is proved in Section \ref{s:IntParts}.
Note that there are m-dissipative extensions of $\A$ (for example, $\A^\Dr$ in Section \ref{s:NtDKrRes}) that cannot be associated with a generalized impedance boundary condition \eqref{e:Imp}.

\subsection{Neumann-to-Dirichlet maps and Krein-type resolvent formulae}
\label{s:NtDKrRes}

For a rigorous randomization of the boundary condition \eqref{e:mdisBC}, which will be done  in Section \ref{s:RBC}, we need two analytic tools,  Neumann-to-Dirichlet maps (depending on the spectral parameter $\la$) and Krein-type resolvent formulae for m-dissipative acoustic systems. 

The Neumann-to-Dirichlet maps (in short, NtD-maps) and their dual maps, Dirichlet-to-Neumann maps (in short, DtN-maps), were intensively studied for various PDEs and have numerous applications, e.g., to  spectral theory and inverse problems.
However, the author was not able to find in the mathematical literature a reference for DtN- or NtD-maps associated with 1st order acoustic systems that would treat general uniformly positive $L^\infty$-coefficients  as in our case (see the references and a discussion in Section \ref{s:DtNdis}).  
That is why we adapt to the m-boundary tuple settings the results of Derkach \& Malamud  on abstract analogues of DtN maps, which are called Weyl M-functions and were introduced for the case of boundary triples in \cite{DM91,DM95} (see also Section \ref{s:DtNKrF} below).

In order to define the NtD-map, let us denote by $\A^\Nr$ (by $\A^\Dr$) the selfadjoint acoustic operator corresponding to the Neumann boundary condition $\gan (\vbf) = 0$ (resp., to the Dirichlet boundary condition $\ga_0 (p) = 0$). 

\begin{thm}[Neumann-to-Dirichlet map] \label{t:NtD}
The exists a unique $\Lc (H^{-1/2} (\paD),H^{1/2} (\paD))$-valued holomorphic function $\la \mapsto M_{\NtD} (\la)$, which we call \emph{the NtD map}, that is defined for all complex numbers
 $\la \in \rho (\A^\Nr) $ in  such a way that 
  the equality 
\[
M_{\NtD} (\la) \gan (\vbf) =  -\ii \ga_0 (p) 
\]
holds for every solution $\{\vbf,p\}$
to the time-harmonic acoustic system  
\begin{gather*} 
  \ab^{-1} \gradm  p = \ii \la \vbf , \quad \beta^{-1} \Div \vbf = \ii \la p , \quad \text{ where $\vbf \in \HH (\Div,\D)$ and $p \in H^1 (\D) $.}
\end{gather*}
Moreover, the function $M_\NtD$ has the following properties: 
\item[(i)] $M_\NtD (\overline{\la}) = (M_\NtD (\la))^\cross$  and $(\im \la) \, \im \<  M_\NtD (\la) f,f\>_{L^2 (\pa D)} \ge 0$ for all $f \in H^{-1/2} (\paD)$ and $\la \in \rho ( \A^\Nr)$.
\item[(ii)] $M_{\NtD} (\la) \in \Hom (H^{-1/2} (\paD),H^{1/2} (\paD))$ for all $\la \in \rho (\A^\Dr) \cap \rho (\A^\Nr)$.
\end{thm}

This theorem is proved in Section \ref{s:DtN}.

\begin{rem}[Dirichlet-to-Neumann map]\label{r:DtN}
For $\la \in \rho (\A^\Dr) \cap \rho (\A^\Nr)$,
Theorem \ref{t:NtD} implies the existence of $(M_{\NtD} (\la))^{-1} \in \Hom (H^{1/2} (\paD),H^{-1/2} (\paD))$. \emph{The DtN map} can be  defined by 
$
M_{\DtN} (\la):= - (M_{\NtD} (\la))^{-1} $ first for all $\la \in \rho (\A^\Dr) \cap \rho (\A^\Nr)$,
and then  can be extended to all $\la  \in \rho (\A^\Dr)$ as a holomorphic $\Lc (H^{1/2} (\paD),H^{-1/2} (\paD))$-valued function. It has the properties analogous to those of $M_{\NtD} $ (see Section \ref{s:DtN} for details).
\end{rem}

We use NtD-maps via the following generalization of the Krein formula for resolvents.

\begin{thm}[acoustic Krein-type resolvent formula] \label{t:aKrF}
Let $K$ be a contraction in $L^2 (\pa \D)$, and let $\wh A$ be the acoustic 
operator associated with the boundary condition \eqref{e:mdisBC}.
Then:
\begin{itemize}
\item[(i)] A complex number $ \la \in \rho (\A^\Nr)$ is in the spectrum $ \si (\wh \A)$ of $\wh A$ if and only if 
\begin{gather*} \label{e:E0E1}
\text{$E_1 V^{-1} M_{\NtD} (\la) + E_0 V^\cross \ \not \in \ \Hom (H^{-1/2} (\pa \D), L^2 (\pa \D)) $},
\end{gather*}
 where 
$ E_0 = I_{L^2 (\pa \D)} - K$ and $ E_1 = -\ii (I_{L^2 (\pa \D)}+K) $.
\item[(ii)]  For $\Psi = \{\vbf,p\} \in \HH (\Div,\D) \times  H^1 (\D) $, we put $\wh \ga_0  \Psi := \ga_0 (p) $, \quad $\wh \ga_\n \Psi := \gan (\vbf)$, and define a holomorphic function $\ga_\Dr  :\rho (\A^\Nr)
\to \Lc (\LL^2_{\ab,\beta} (\D),  H^{1/2} (\pa \D))$ by $\ga_\Dr (\la) := \wh \ga_0  (\A^\Nr - \lambda  )^{-1}$.
Then the following formula is valid  for all $\la \in \rho (\wh \A)  \cap \rho (\A^\Nr) $ and $\Psi, \Phi \in \LL^2_{\ab,\beta} (\D)$:
\begin{multline}  \label{e:KFDtN}
( [\A^\Nr - \la ]^{-1} \Psi | \Phi )_{\LL^2_{\ab,\beta}}   - ( [\wh \A - \la ]^{-1} \Psi | \Phi  )_{\LL^2_{\ab,\beta}}  
\\ = \left\< [E_1 V^{-1} M_{\NtD} (\la) + E_0 V^\cross]^{-1} E_1 V^{-1} \ga_\Dr (\la)  \Psi ,  \ga_\Dr (\overline{\la}) \Phi \right\>_{L^2 (\pa \D)},
 \end{multline}
where $
\<\cdot,\cdot\>_{L^2 (\pa \D)}$ is understood as the ${L^2 (\pa \D)}$-paring of $H^{-1/2} (\pa \D)$ and $H^{1/2} (\pa \D)$.
\end{itemize}
\end{thm}

This theorem is proved in Section \ref{s:KreinF}, where we combine the m-boundary tuple method  of \cite{EK22} with the abstract Krein-type formulae of
 \cite{DM91,DM95} (see also the monographs \cite{DM17,BHdS20}).

\subsection{Random m-dissipative boundary conditions for acoustic system}
\label{s:RBC}

In the sequel, $(\Om, \Fc,\PP)$ is the underlying complete probability space.  Under random variables we understand $(\Om,\Fc)$-measurable functions from $\Om$ to $\CC$. Random vectors are functions $h:\Om \to \Yf$ with values in a separable Banach space $\Yf$ such that $\{ \om \in \Om : h (\om) \in B\} \in \Fc$ for any Borel set $B$ in $\Yf$.

Let $\Xf$ be a separable Hilbert space. By $\dim \Xf $ we denote its dimensionality. In the case $\dim \Xf = \infty$, the simplest type of an operator randomness can be introduced for functions $T : \om \mapsto T_\om$, $T:\Om \to \Lc (\Xf)$, in the following way.
We say that an $\Lc (\Xf)$-valued function $T$ (defined on an event $\Om_1 \subseteq \Om$ of probability 1)  is a  \emph{random bounded operator in $\Xf$} if $(Tf|g)_{\Xf}$ is a random variable for all $f,g \in \Xf$ 
(this definition corresponds to the class $\mathbf{L} (\Om,\Xf)$ of \cite{S84}). 

Let $T$ be a random bounded operator in $\Xf$. Then (see, e.g., \cite{BR72,DZ92}) 
\begin{equation} \label{e:|T|}
\text{$\| T\|$ is random variable with the values in $\overline{\RR}_+ := [0,+\infty)$.} 
\end{equation}
A random bounded operator $T$ is said to be a \emph{random  contraction} if $\| T \| \le 1$ a.s., and is said to be  a \emph{random unitary operator} if $T$ is a.s. a unitary operator.

Random contractions will be used for the description of random boundary conditions, see Theorem \ref{t:randM-dis}. 
It is difficult to apply the definition given above  to randomized unbounded operators, in particular, to acoustic operators $\wh \Ac$ associated with randomized boundary conditions, because such operators   have domains  depending on $\om \in \Om$. 
However, under certain conditions the definition of random bounded operator can be applied to resolvents.

Namely, consider an operator-valued function $T: \om \mapsto T_\om$ such that $T_\om : \dom T_\om \subseteq \Xf \to \Xf$ is a selfadjoint operator for a.a. $\om \in \Om$. Then  $T$ is called a \emph{random-selfadjoint operator}  if $(T - z)^{-1}$ is a random bounded operator for every $z \in \CC \setminus \RR$.
This is an adaptation of the definition from the monograph \cite{PF92} to our situation. A natural modification for the  m-dissipative case is the following.

\begin{defn}[random-m-dissipative operator]\label{d:rmDis}
Assume that an operator-valued function $T: \om \mapsto T_\om$ is such that $T_\om : \dom T_\om \subseteq \Xf \to \Xf$ is a.s. an m-dissipative operator in $\Xf$. We say that $T$ is a \emph{random-m-dissipative operator} if its resolvent $(T - z)^{-1}$ is a random bounded operator for every $z \in \CC_+$.
\end{defn}

\begin{thm} \label{t:randM-dis} 
Let $K$ be a random contraction in $L^2 (\pa \D)$. Let us fix a certain deterministic $V \in \Hom (L^2 (\pa \D), H^{1/2} (\pa \D))$ (like in Theorem \ref{t:AM-dis} or in Lemma \ref{l:VV*}).
Then:
\begin{itemize}
\item[(i)] The acoustic operator $\wh \A $ 
associated 
 with the randomized boundary condition 
 \begin{gather} \label{e:mdisBCR}
(K+I_{L^2 (\pa \D)}) V^{-1} \ga_0 (p)  + (K-I_{L^2 (\pa \D)}) V^\cross \gan (\vbf) = 0 
\end{gather}
is a random-m-dissipative operator.
\item[(ii)] If, additionally, $K$ is a.s. unitary, then $\wh \A$ is a random-selfadjoint operator.
\end{itemize}
\end{thm}

We prove this theorem in Section \ref{s:RandProof} combining the Hanš theorem on random inverse operators (see Proposition \ref{p:HTh}) with the Krein-type resolvent formula of Theorem \ref{t:aKrF}.

For $0 <  p \le \infty$, we denote by $\Sf_p = \Sf_p (\Hs)$  the Schatten-von-Neumann ideals of compact operators in a separable complex Hilbert space $\Hs$ (if the choice of the space $\Hs$ is clear from the context we omit $\Hs$ in the notation $\Sf_p$).  The operator trace $\Tr \in \Lc(\Sf_1 (\Hs), \CC )$ is well-defined on the trace class $ \Sf_1 $ \cite{DZ92,Kato}.
The class $\Sf_2 $ of Hilbert–Schmidt operators is a separable Hilbert space with the inner product $(R|T)_{\Sf_2} := \Tr (T^*R)$ and with the norm $\| T \|_{\Sf_2} = (\Tr (T^*T))^{1/2}$.

Let $\{e_j\}_{j\in \NN}$ be a certain deterministic  orthonormal basis in $L^2 (\pa \D)$. Then $\{e_j \otimes e_k\}_{j,k\in \NN}$ is an orthonormal basis in $\Sf_2 (L^2 (\pa \D))$.
Note that
\begin{gather} \label{e:SfRandV}
\text{a $\Sf_2 (L^2 (\pa \D))$-valued random vector is a random bounded operator in $L^2 (\pa \D)$  \cite{DZ92}}. 
\end{gather}
In order to see this, one considers random variables $\Tr \left( [e_1 \otimes g] T [f \otimes e_1] \right)$ for $f,g \in L^2 (\pa \D) $.

\begin{ex} \label{e:Gaussian0} 
Examples of random contractions for Theorem \ref{t:randM-dis} can be built from well-known distributions.
For  instance, let $\Kc_0 \in \Lc (L^2 (\pa \D))$ be a deterministic operator. Let $J$ be a
Gaussian $\Sf_2 (L^2 (\pa \D))$-valued random vector (in a narrow or in a wide sense \cite{VK97}) on a certain complete probability space $(\wh \Om, \wh \PP, \wh \Fc)$. By \eqref{e:SfRandV}, $\wh K :=\Kc_0+J$ is a random bounded operator. Assume additionally that $\Kc_0$ and $J$ are such that the event $\Om := \{\| \wh K \| \le 1 \}$ is of positive probability. Then conditional probabilities $\PP (E) := \wh \PP (E|\Om) $,
$E \in \wh \Fc$, are well-defined. Taking $\Fc := \{ E \in \wh \Fc : E \subseteq \Om\}$, one obtains a new probability space $(\Om,\Fc,\PP)$ and sees that $K:= \wh K\uph_{\Om}$ is a random contraction on this space. 
\end{ex}


\begin{ex}[random matrices] \label{ex:RandomMatrix}
Let us fix certain $n \in \NN$ and $c \in \TT := \{z\in \CC: |z|=1\}$.
Let $\Hs_n:=\Span \{e_j\}_{j=1}^n$ be the  linear hull of $\{e_j\}_{j=1}^n$.
Consider a unitary circular ensemble  $U^{[n]}$ in $\Hs_n$ \cite{BO01} and the identity operator $I_{\Hs_n^\perp}$ in the orthogonal complement $\Hs_n^\perp := L^2 (\pa \D) \ominus \Hs_n$.
Then the orthogonal sum $\Kc^{[n]} = U^{[n]} \oplus c I_{\Hs_n^\perp}$ is random unitary operator in $L^2 (\pa \D)$, which defines  a random-selfadjoint acoustic operator $\wh A^{[n]}$ via Theorem \ref{t:randM-dis} (with $K=\Kc^{[n]}$). In the logic of random matrix theory, the point of interest is the possibility to pass to reasonable 
 weak limits of resolvents $(\wh \A^{[n_k]} - \la)^{-1}$
for subsequences $n_k \to \infty$.
\end{ex}

Our attention will be concentrated on  more explicit constructions of random contractions. Let 
 $\Ic = \NN$ or $\Ic = \NN^2$ be a countable set of indices.
Let $D_i  \in \Lc (L^2 (\pa \D))$, $i \in \Ic$, and $\Kc_0 \in \Lc (L^2 (\pa \D)) $ be deterministic operators. Assume that  complex random variables $\xi_i$, $i \in \Ic$, are such that
the sum $\sum_{i \in \Ic} \xi_i D_i$ converges to a certain $\Kc_1: \Om \to \Lc (L^2 (\pa \D))$
a.s. w.r.t. the operator norm or w.r.t. the $\Sf_2$-norm (if the order of summation is defined one can consider also the weak operator convergence, see Example \ref{ex:normal} below).
Then $\Kc_1$ and $K=\Kc_0+\Kc_1$ are random bounded operators. We consider several ways to ensure that $K$ is a random contraction, and so, that a random-m-dissipative operator $\wh \A$ of Theorem \ref{t:randM-dis} is well-defined.

\begin{ex}[shifted $\Sf_2$-distributions] \label{ex:K0+S2}
The double sequence $\{D_{j,k}\}_{j,k \in \NN} = \{e_j \otimes e_k\}_{j,k\in \NN}$ is an orthonormal basis in $\Sf_2 (L^2 (\pa \D))$. Consider a deterministic $\Kc_0 \in \Lc (L^2 (\pa \D))$ and a double sequence $\{r_{j,k}\}_{j,k \in \NN} \subset \overline{\RR}_+$ of nonnegative numbers such that 
$\|\Kc_0\| + (\sum_{j,k\in \NN} r_{j,k}^2)^{1/2} \le 1$. Let $\xi_{j,k}$, $j,k \in \NN$, be random variables with values in the complex discs $\{z \in \CC:|z|\le r_{j,k}\}$.
Then, with probability 1, the sum $\sum_{j,k\in \NN} \xi_{i,j} (e_j \otimes e_k) $ converges absolutely in $\Sf_2$  to a random $\Sf_2$-vector $J$ with $\|J\|^2 \le \|J\|_{\Sf_2}^2 \le (1-\|\Kc_0\|)^2 $.
Using \eqref{e:SfRandV}, we see that  $K=K_0 +J$ is a random contraction. 
\end{ex}

\begin{ex}[quasi-uniform distributions] \label{ex:quasi-uniform}
Let $\Kc_0$ and $D_j$, $j\in \NN$, be deterministic contractions in $L^2 (\pa \D)$.
Consider a sequence of nonnegative numbers $\{a_j\}_{j \in \NN} \subset \overline{\RR}_+$ such that $\|\Kc_0\|+\sum_{j=1}^\infty a_j \le  1$.  Let $\xi_j$, $j \in \NN$, be independent random variables uniformly distributed in the intervals $(-a_j,a_j)$ (or in the  complex discs $\{z \in \CC : |z|<a_j\}$).  
Then $K = \Kc_0 + \sum_{j=1}^\infty \xi_j D_j $ is a random contraction. 
\end{ex}

Example \ref{ex:quasi-uniform} is based on  the convexity of a closed unit ball $\overline{\BB}_1 (0;\Yf) = \{y\in \Yf : \|y\|_{\Yf} \le 1\}$ in a Banach space $\Yf$. The same idea is employed  in the following more flexible construction, which will be useful in Section \ref{s:PartComp}.

\begin{ex}
\label{ex:ncq-uD}
Let $\Kc_0 \in \Lc (\Hs)$ be a contraction.  We say that an operator $D \in \Lc (\Hs)$ is an admissible direction from $\Kc_0$ if there exists a constant $c>0$ such that $\|\Kc_0+ c   D \| \le 1$. If this is the case,
then the convex hull $S (\Kc_0,D)$ of all complex numbers $z$ such that $\|\Kc_0+ z   D \| \le 1$ contains 
a segment connecting $0$ and $c$. Clearly, $\|\Kc_0+ z   D \| \le 1$ for every $z \in S (\Kc_0,D)$.
We take $\Hs = L^2 (\pa \D)$ and consider  directions $ D_j $, $j \in \NN$, admissible from a contraction $\Kc_0 \in  \Lc (\Hs)$. Assume that constants  $b_j>0$ satisfy  
$\sum_{j \in \NN} b_j \le 1$. Let $\xi_j$ be a   
complex random variable with values in $S (\Kc_0,D_j)$ for each  $j \in \NN$.
Then the convexity of  $\overline{\BB}_1 \left( 0;\Lc (L^2 (\pa \D)) \right)$ implies that the sum  
$K = \Kc_0 + \sum_{j \in \NN} b_j \xi_j   D_j $ a.s. converges in the operator norm to a random contraction $K$.
\end{ex}

\subsection{Discrete spectra and the partial compactness of resolvents}
\label{s:PartComp}

The relatively simple examples of random-m-dissipative acoustic operators $\wh A$ in the previous subsection typically do not satisfy the property of the partial  resolvent compactness  (A3) introduced in Section \ref{s:i}.
Rigorous versions of this statement are given by Remarks  \ref{r:non-comp1}-\ref{r:ShiftedConvex}.
They follow from Theorem \ref{t:AM-disComp}, which is the main result of this subsection and which gives a  characterization of property (A3) in terms of random contractions $K$.
 

An eigenvalue $\la$ of an operator $T:\dom T \subseteq \Xf \to \Xf$ is called \emph{isolated} if $\la$ is an isolated point of the spectrum $\si (T)$ of $T$. The \emph{discrete spectrum} $\si_\disc (T)$ of $T$ is the set of isolated eigenvalues of $T$ with finite algebraic multiplicities (see \cite{Kato} and \cite[Vol.4]{RS}). We say that $T$ has  \emph{purely discrete spectrum} if $\si(T) = \si_\disc (T)$. The closed set $\si_\ess (T) := \si (T) \setminus \si_\disc (T)$ is called an \emph{essential spectrum} of $T$  \cite[Vol.4]{RS}.
Assume now that $\la_0$ belongs to the resolvent set $\rho (T)$ of $T$ and the resolvent $(T-\la_0 )^{-1}$ at $\la_0$ is a compact operator. 
Then $(T-\la)^{-1}$ is compact for every $\la \in \rho (T)$; in this case, it is said that $T$ is an \emph{operator with compact resolvent} (this definition assumes $\rho (T) \neq \varnothing$). 
Note that an operator $T$ with compact resolvent has purely discrete spectrum \cite{Kato} and, in particular, cannot have an infinite-dimensional kernel.

The kernel $
\ker \Div_0 = \HH_0 (\Div 0,\D) = \{ \ubf \in \HH_0 (\Div,\D) \ : \ \Div \ubf  = 0 \}$ of the operator $\Div_0$ is a closed subspace of $\LL^2_{\ab} (\D)$.
The following  orthogonal decomposition takes place
\begin{gather} \label{e:La}
\LL^2_{\ab} (\D) = \HH_0 (\Div 0,\D) \oplus \ab^{-1} \gradm H^1 (\D) , 
\end{gather}
where $ \ab^{-1} \gradm H^1 (\D) := \{ \ab^{-1} \nabla p : p \in H^1 (\D)\}$ equipped with the norm of  $\LL^2_{\ab} (\D)$ is a Hilbert space. Indeed, 
the adjoint operator to $\ab^{-1} \gradm : H^1 (\D) \subset L^2 (\D) \to \LL^2_{\ab} (\D)$ equals $\Div_0: \HH_0 (\Div,\D) \subset \LL^2_{\ab} (\D) \to L^2 (\D) $. Since the normed space $( \ab^{-1} \gradm H^1 (\D) , \| \cdot \|_{\LL^2_{\ab}})$ is complete (see \cite[Section 7]{L13}), one obtains  \eqref{e:La} from \cite[Theorem IV.5.13]{Kato}.

We see that the space $ \LL^2_{\ab,\beta} (\D) = \LL^2_{\ab} (\D) \oplus L^2_\beta (\D) $ 
has the (orthogonal) decomposition
\begin{gather} \label{e:LabDec}
\LL^2_{\ab,\beta} (\D) = \HH_0 (\Div 0,\D) \oplus \GG_{\ab,\beta} , 
\end{gather}
where $ \GG_{\ab,\beta} := \ab^{-1} \gradm H^1 (\D) \oplus L^2_{\beta} (\D) $ and $\HH_0 (\Div 0,\D)$ are perceived as closed subspaces of  $\LL^2_{\ab,\beta} (\D)$ and are equipped with the norm of 
 $\LL^2_{\ab,\beta} (\D)$.
In these settings, $\HH_0 (\Div 0,\D) = \ker \A$ for the closed symmetric operator $\A$ of \eqref{e:A}, and so,
 $\HH_0 (\Div 0,\D)$ is a reducing subspace of $\A$ (see \cite{AG} and Section \ref{s:SAPart} for basic facts about reducing subspaces). That is, the decomposition \eqref{e:LabDec} reduces $\A$ to the orthogonal sum
$\A = 0 \oplus \A |_{\GG_{\ab,\beta}}$, where the part 
$\A|_{\HH_0 (\Div 0,\D)}$ of $\A$ in the space $\HH_0 (\Div 0,\D)$ is the zero operator.

 Since $\HH_0 (\Div 0,\D)$ is infinite-dimensional, any of the m-dissipative extensions $\wh \A$ of $\A$ described in Theorem \ref{t:AM-dis} does not possess a compact resolvent. 
 From the point of view of the time-dependent acoustic equation, the elements of $\ker \A$ are stationary solutions that represent the static background vector field. They can be disregarded in the study of the dynamical properties (see Section \ref{s:El}).
 One of the ways to use the compactness of resolvent for the study of $\si (\wh \A)$ is to consider the part of $\wh \A$ in the reducing subspace $\GG_{\ab,\beta}$.

\begin{lem}
\label{l:AcReducing}
(i) The decomposition \eqref{e:LabDec}  reduces $\A^*$  to $\A^* = 0 \oplus \A^* |_{\GG_{\ab,\beta}} = 0 \oplus (\A |_{\GG_{\ab,\beta}})^*$ and reduces every dissipative extension $\wt \A$ of $ \A$ to $\wt \A = 0 \oplus \wt \A |_{\GG_{\ab,\beta}}$.
\item[(ii)] Let an acoustic operator $\wh \A$ be defined as in Theorem \ref{t:randM-dis}. The decomposition \eqref{e:LabDec} reduces $\wh \A$ a.s.  to $\wh \A = 0 \oplus \wh \A |_{\GG_{\ab,\beta}}$, where $\wh \A |_{\GG_{\ab,\beta}}$ is a random-m-dissipative operator in 
$\GG_{\ab,\beta}$.
\end{lem}

\begin{proof} The proof of statement (i) can be  obtained from \cite[Theorem 46.5]{AG}  by an adaptation of the arguments used for Maxwell operators in \cite[Remark 2.2]{EK22}  (see also  Proposition \ref{p:reductionA} below). 
The stochastic statement (ii) follows from the deterministic statement (i).
\end{proof}

Assume that $K:\om \mapsto K_\om$ is a random contraction in $L^2 (\pa \D)$ and $\wh \A$ is a random-m-dissipative acoustic operator defined as in Theorem \ref{t:randM-dis}. 
Compactness criteria for operators  imply that $\{\om \in \Om: K+I \in \Sf_\infty (L^2 (\pa \D))\} \in \Fc$,  i.e., this set of $\om$ is an event (see Lemma \ref{l:CompR}).

\begin{thm}\label{t:AM-disComp}
(i) The part $\wh \A|_{\GG_{\ab,\beta}}$ of  $\wh \A$ in $\GG_{\ab,\beta}$ has a compact resolvent with probability 
$\PP \{K+I \in \Sf_\infty (L^2 (\pa \D))\}$. In particular, the resolvent of $\wh \A|_{\GG_{\ab,\beta}}$ is a.s. compact 
if and only if 
\begin{gather} \label{e:K+Icomp}
\text{the operator $K+I_{L^2 (\pa \D)}$ is a.s. compact.}
\end{gather} 
\item[(ii)] If \eqref{e:K+Icomp} holds, then  a.s.
$
\si (\wh \A|_{\GG_{\ab,\beta}}) = \si_\disc (\wh \A|_{\GG_{\ab,\beta}}) $ and  $\si (\wh \A) = \si_p (\wh \A)$.
\item[(iii)] Assume additionally that $K$ is a.s. unitary.
Then $\wh \A|_{\GG_{\ab,\beta}}$ is a random-selfadjoint operator in $\GG_{\ab,\beta}$. Moreover, 
$
\PP \{ \si (\wh \A|_{\GG_{\ab,\beta}}) = \si_\disc (\wh \A|_{\GG_{\ab,\beta}}) \} = 
\PP \{K+I \in \Sf_\infty (L^2 (\pa \D)) \} .
$
\end{thm}

This theorem is proved in Section \ref{s:ReducingA}.

\begin{rem} \label{r:DiscDet}
The spectral theorem for normal operators  (see, e.g., \cite{AG}) implies that a normal operator $T:\dom T \subseteq \Xf \to \Xf$ has a compact resolvent if and only if $\si (T) = \si_\disc (T)$. In particular, this statement is valid for selfadjoint operators. For a general (not necessarily normal) operator $T$, the resolvent compactness  implies $\si (T) = \si_\disc (T)$, but not vice versa.
\end{rem}

\begin{rem} \label{r:DiscDet0}
Theorem \ref{t:AM-disComp} is equivalent to its deterministic analogue (see  Theorem \ref{t:AM-disCompDet}). In particular, $ \{\om \in \Om : \text{the resolvent of $\wh \A|_{\GG_{\ab,\beta}}$ is compact} \} = \{\om \in \Om : K+I \in \Sf_\infty \} \in \Fc$.
If $K$ is a random-unitary operator, 
$
\{\om \in \Om: K+I \in \Sf_\infty \} = \{ \om \in \Om: \si (\wh \A|_{\GG_{\ab,\beta}}) = \si_\disc (\wh \A|_{\GG_{\ab,\beta}}) \} \in \Fc .
$
\end{rem}

\begin{cor} \label{c:<1}
$\PP \{ \text{the resolvent of $\wh \A|_{\GG_{\ab,\beta}}$ is compact} \} \le \PP \{\|K\| =1\}$
\end{cor}
\begin{proof}
If $\| K_\om \|<1$, then  $K_\om +I \in \Hom (L^2 (\pa \D))$ and $K_\om +I \not \in \Sf_\infty$. Theorem \ref{t:AM-disComp} (i) and Remark \ref{r:DiscDet} complete the proof.
\end{proof}

\begin{rem} \label{r:non-comp1}
Corollary \ref{c:<1} implies that, for an acoustic operator $\wh \A$ associated with a random contraction $K$ constructed as in  Example \ref{ex:quasi-uniform}, the resolvent of $\wh \A|_{\GG_{\ab,\beta}}$ is compact with probability 0. In the case of Example \ref{ex:K0+S2}, Theorem \ref{t:AM-disComp} (i) yields that the resolvent of $\wh \A|_{\GG_{\ab,\beta}}$ is a.s. compact only in the deterministic case where $K=\Kc_0 = -I$ for a.a. $\om \in \Om$. In the case of Example \ref{e:Gaussian0}, it follows from Theorem \ref{t:AM-disComp} (i) that the a.s. compactness of the resolvent of $\wh \A|_{\GG_{\ab,\beta}}$ is equivalent to  $\Kc_0 + I \in \Sf_\infty (L^2 (\pa \D))$.
\end{rem}

\begin{rem} \label{r:non-comp2}
Let $n \in \NN$ and $\wh \A^{[n]}$ be a random-selfadjoint acoustic operator associated with the random contraction $\Kc^{[n]} = U^{[n]} \oplus c I_{\Hs_n^\perp}$ of Example \ref{ex:RandomMatrix}, where $U^{[n]}$ a unitary circular ensemble. Then $\PP \{\si (\wh \A|_{\GG_{\ab,\beta}}) = \si_\disc (\wh \A|_{\GG_{\ab,\beta}})\} >0$  if and only if $c=-1$. If this is the case, then $\si (\wh \A|_{\GG_{\ab,\beta}}) = \si_\disc (\wh \A|_{\GG_{\ab,\beta}})$ with probability 1. This statements follow from Theorem \ref{t:AM-disComp}.
\end{rem}

\begin{rem} \label{r:ShiftedConvex}
Let $K = \Kc_0 + \sum_{j \in \NN} b_j \xi_j   D_j $ be a random contraction constructed as in Example \ref{ex:ncq-uD}. Assume additionally  that the directions $D_j $ admissible from $\Kc_0$ are  compact for all $j \in \NN$.
Let $\wh \A$ be the associated acoustic operator. Then Theorem \ref{t:AM-disComp} (i) implies that the resolvent of $\wh \A|_{\GG_{\ab,\beta}}$ is a.s. compact if and only if $\Kc_0 + I \in \Sf_\infty $.
\end{rem}

Summarizing the statements of Theorem \ref{t:AM-disComp} and Remarks \ref{r:non-comp1}-\ref{r:ShiftedConvex}, one infers that only random contractions $K$ built as compact perturbations of $\Kc_0 = - I_{L^2 (\pa \D)}$ lead to acoustic operators $\wh \A|_{\GG_{\ab,\beta}}$ with a.s. compact resolvents. That is why the case $\Kc_0 = - I_{L^2 (\pa \D)}$ in Example \ref{ex:ncq-uD} is special.

\begin{lem} \label{l:dir} 
An operator $D  \in \Lc (\Xf) $ is an admissible direction from $(- I)$  if and only if there exists a constant $c_1>0$ such that $c_1 (\im D)^2 \le \re D$. In particular,  every direction admissible from $(- I_\Xf)$   is an accretive operator.
\end{lem}
\begin{proof}
From the convexity of $\overline{\BB}_1 ( 0;\Lc (\Xf) )$, one gets that $\|c   D -  I\| \le 1$ holds for a certain  $c>0$ exactly when 
$\|s   D -  I\| \le 1$ for all $s \in (0,c]$, and exactly when $\|s   D^* -  I\| \le 1$ for all $s \in (0,c]$.
Further, $\|s   D -  I_\Xf\| \le 1$ for a certain $s>0$ is equivalent to 
$2 \re D - s D^* D  \ge 0$.  Since $D^* D \ge 0$, the admissibility of $D$ from $(- I)$ implies that $\re D \ge 0$, i.e., that $D$ is accretive.

Besides, $D  $ is an admissible direction from $(- I)$ if and only if $\re D - s (D^* D + D D^*) \ge 0$ for small enough $s>0$. The spectral decomposition of the bounded selfadjoint operator $\re D$ implies that $s (\re D)^2 \le \re D$ for small enough $s>0$.
Since $\frac 12 (D^* D + D D^*) = (\re D)^2+(\im D)^2$, we see that
$\re D - s (D^* D + D D^*) \ge 0$ is valid for small enough $s>0$ if and only if 
$\re D - s (\im D)^2 \ge 0 $ for small enough $s>0$. This completes the proof.
\end{proof}


Clearly, operators $e_j \otimes e_j$ are admissible direction from $(-1) I_{\Lc (\pa \D)}$. In the case  $j \neq k$,  Lemma \ref{l:dir} shows that $c (e_j \otimes e_k)$ is not an admissible direction for any $c \in \CC \setminus \{0\}$. In other words, if we want to build
$K$ as an infinite random matrix with independent random entries using Example \ref{ex:ncq-uD} 
and $D_j$ of the form $e_i \otimes e_j$, then 
only the  diagonal elements are allowed to be nonzero in this matrix. 
However, the diagonal operators $K = \sum_{j \in \NN} \xi_j (e_j \otimes e_j)$ with complex random variables $\xi_j$ are a.s. normal and can be studied without additional restrictions of Example \ref{ex:ncq-uD}.

\begin{ex}[diagonal random contractions] \label{ex:normal}
Let us fix an arbitrary deterministic orthonormal basis $\{e_j\}_{j \in \NN}$ in $L^2 (\pa \D)$.
Let $\xi_j $, $j \in \NN$, be complex random variables. First, we define
a randomized operator $\wt K$ by $\wt K \sum c_j e_j = \sum \xi_j c_j e_j$ on the set   $\dom \wt K$ of all finite linear combinations $\sum c_j e_j$ with coefficients $c_j \in \CC$.
It is clear that $\wt K$ can be extended to a random bounded operator $K$ if and only if $\xi := \{\xi_j\}_{j \in \NN}$ belongs to $ \ell^\infty (\NN)$ with probability 1, and that this $K$ is a random contraction if and only if $\|\xi  \|_{\ell^\infty} \le 1$ with probability 1.
Assume now that $\|\xi  \|_{\ell^\infty} \le 1$ a.s., and so, a random-m-dissipative acoustic operator $\wh \A$ is associated via Theorem \ref{t:randM-dis} with the random contraction $K$. 
Theorem \ref{t:AM-disComp} implies that
\begin{gather} \label{e:Pto1}
\text{the resolvent of $\wh \A|_{\GG_{\ab,\beta}}$ is compact with the probability $\textstyle \PP \{ \lim_{j\to\infty} \xi_j = -1 \}$. }
\end{gather}
\end{ex}

\begin{rem} \label{r:RandSA}
In Example \ref{ex:normal}, $K$ is a random-unitary operator if and only if distributions of all $\xi_j $ are be supported on the unit circle $\TT$. In this case,  $\wh \A$  is a random-selfajoint operator. Theorem \ref{t:AM-disComp} implies that $\PP \{ \si (\wh \A|_{\GG_{\ab,\beta}}) = \si_\disc (\wh \A|_{\GG_{\ab,\beta}}) \} = \PP \{ \lim_{j\to\infty} \xi_j = -1 \}$.
\end{rem}

If one takes all $\xi_j$ in  Example \ref{ex:normal} to be  independent and identically distributed (i.i.d.) with the uniform distribution in the disc $\{z \in \CC : |z| < 1\}$, the corresponding random contraction $K$ can be seen as one of possible substitutions to the non-existing uniform distribution on the ball $\overline{\BB}_1 (0; \Lc (L^2 (\pa \D)))$ of contraction. However, in this case, $\wh \A|_{\GG_{\ab,\beta}}$ has a compact resolvent with probability 0. This can seen from the following more general result.

\begin{cor} \label{c:normal0}
In Example \ref{ex:normal}, assume that $\xi_j$, $j\in \NN$, are 
i.i.d. random variables with the distribution supported in $\{z \in \CC : |z| \le 1\}$ such that $\PP \{\xi_1 = -1\} \neq 1$. Then the resolvent of $\wh \A|_{\GG_{\ab,\beta}}$  is a.s. not compact.
\end{cor}

\begin{proof}
There exists $\de>0$, such that $\PP \{|\xi_1+1| \le \de\} = c <1 $. So, for each $n \in \NN$, $P_n := \PP \{|\xi_j+1| \le \de \ \forall j \ge n \} = \prod_{j=n}^\infty c = 0$. Thus, 
$\PP \{ \lim_{j\to\infty} \xi_j = -1 \} \le \lim_{n \to \infty} P_n = 0$. Statement \eqref{e:Pto1} completes the proof.
\end{proof}

The following result is a preparation for Section \ref{s:RImp}, where a connection with the Weyl's asymptotic law for eigenvalues of $\De^\paD$ is discussed.

\begin{cor} \label{c:normal}
In Example \ref{ex:normal}, let $\xi_j$, $j\in \NN$, be
independent random variables with distributions supported in $\{z \in \CC : |z| \le 1\}$. Consider the event 
\[
\text{$\Om_\comp := \{ \wh \A|_{\GG_{\ab,\beta}} \text{ has a compact resolvent}\}$ \ (see Remark \ref{r:DiscDet0}).}
\]
 Then $\PP (\Om_\comp)$   is either 0, or 1. Besides, $\PP (\Om_\comp) =1$ if and only if for every $\de>0$ there exists  $N=N(\de) \in \NN$ such that $0<\prod_{j=N}^{\infty} \PP \{|\xi_j +1| < \de\}$. 
\end{cor}

\begin{proof} Combining  \eqref{e:Pto1} with the  Kolmogorov's zero-one law, one sees  that $\PP (\Om_\comp)$ is either 0, or 1. 
 For every $\de>0$, let us consider the $[-\infty, +\infty]$-valued random variable 
 \[
 \wt n (\de) := \sup \{j \in \NN: |\xi_j +1| \ge \de\}, \text{ where $\sup \emptyset := - \infty$}.
 \]
 Since $\PP \{\wt n (\de) < n \} = \prod_{j=n}^{\infty} \PP \{|\xi_j +1| < \de\} $,
we see that $\PP \{\wt n (\de) < +\infty\} = 1$ if and only if 
$0 < \prod_{j=n}^{\infty} \PP \{|\xi_j +1| < \de\} $
 for $n$ larger or equal than a certain $N=N(\de) \in \NN$. By \eqref{e:Pto1} and Remark \ref{r:DiscDet0}, $\Om_\comp  = \bigcap_{\de >0} 
\{\wt n (\de) < +\infty\}$. This completes the proof.
\end{proof}

\subsection{Absorbing boundary condition for elliptic second order operators}
\label{s:El}

Recall that the uniformly positive matrix-functions  $\al$ and $\al^{-1}$ belong to $L^\infty (\D, \RR^{d\times d}_\sym)$.
We consider in the Sobolev space $H^1 (\D) $ the semi-norm $\| u \|_{1,\ab^{-1}} = (\ab^{-1} \nabla u|\nabla u)_{L^2 (\D)}^{1/2}$ and define the Hilbert space $(H_{1,\ab^{-1}} (\D), \| \cdot \|_{1,\al^{-1}}) $ as the result of the factorization of the semi-Hilbert space $(H^1 (\D), \| \cdot \|_{1,\al^{-1}} )$ w.r.t. the 1-dimensional subspace $\{c \one : c \in \CC\}$ of constant functions. 

The restriction  of operators $\A$, $\wh \A$, and $\A^*$ to the subspace $\GG_{\ab,\beta}$ considered above (see  Lemma  \ref{l:AcReducing}) allows one to the transform the 1st order acoustic equation 
$ \ii \pa_t \Psi   = \af_{\ab,\beta} \Psi $ into its 2nd order version
\begin{gather} \label{e:Ac2or}
\ii \pa_t \Phi = \bfr_{\ab,\beta} \Phi, \quad \text{ with  } \quad \bfr_{\ab,\beta} : \begin{pmatrix} u  \\ p   \end{pmatrix} \mapsto \ii \begin{pmatrix} 0 & I  \\
\beta^{-1}  \Div  \ab^{-1}  \gradm & 0 \end{pmatrix} \begin{pmatrix} u  \\ p   \end{pmatrix},
\end{gather} 
where the evolution of $\Phi=\{u,p\}$ is considered in the space $H_{1,\ab^{-1}} (\D) \oplus L^2_{\beta} (\D)$.

Let the operator $\gradm_1: u \mapsto  \nabla u$ be defined on the  space $H_{1,\ab^{-1}} (\D)$. The definition of the  norm in $H_{1,\ab^{-1}} (\D)$ implies that the map 
\begin{equation} \label{e:aGrad1}
\text{
$\ab^{-1} \gradm_1 : u \mapsto  \ab^{-1} \nabla u$  is a unitary operator}
 \text{  from $H_{1,\ab^{-1}} (\D)$ to 
$\ab^{-1} \gradm H^1 (\D)$,}
\end{equation}
where $\ab^{-1} \gradm H^1 (\D) = \LL^2_{\ab} (\D) \ominus  \HH_0 (\Div 0,\D) $ is perceived as a closed subspace of $\LL^2_{\ab} (\D)$. 
Note $ u $ in \eqref{e:Ac2or} is the scalar potential of $\ab \vbf$ in the sense $\ab \vbf = - \nabla u $.

\begin{lem} \label{l:U}
For the part $ \A|_{\GG_{\ab,\beta}}$ of the symmetric acoustic operator $\A$, and for its adjoint 
$ \A^*|_{\GG_{\ab,\beta}}$, the following equalities hold 
\begin{align}
& \Uc^{-1} ( \A|_{\GG_{\ab,\beta}}) \Uc = \Bc  \qquad \text{ with } \qquad  \Bc = \ii
\begin{pmatrix} 0 & \Ic_0  \\
\beta^{-1} \Div_0  \ab^{-1} \gradm_1  & 0  \end{pmatrix} , \label{e:B}
\\
\text{and } \quad & \Uc^{-1} (\A^*|_{\GG_{\ab,\beta}}) \Uc = \Bc^* = \ii 
\begin{pmatrix} 0 & \Ic_1  \\
\beta^{-1} \Div  \ab^{-1} \gradm_1  & 0  \end{pmatrix} \label{e:B*}
\end{align}
where 
$\Uc:  \begin{pmatrix} u  \\ p   \end{pmatrix} \mapsto \begin{pmatrix} -\ab^{-1} \nabla u \\ p   \end{pmatrix} $ 
is a unitary operator from $H_{1,\ab^{-1}} (\D) \oplus L^2_{\beta} (\D)$ onto $\GG_{\ab,\beta} $,
the domains of $\Bc$ 
and $\Bc^*$ are given by 
\begin{gather*}
\dom \Bc =\{u \in H_{1,\ab^{-1}} (\D) : \ab^{-1} \nabla u \in \HH_0 (\Div,\D) \} \times  H^1_0 (\D),
\\
 \dom \Bc^* =\{u \in H_{1,\ab^{-1}} (\D) : \ab^{-1} \nabla u \in \HH (\Div,\D) \} \times  H^1 (\D) ,
\end{gather*}
while the operators $\Ic_0: H^1_0 (\D) \subset L^2_{\beta} (\D) \to H_{1,\ab^{-1}} (\D) $ and 
$\Ic_1: H^1 (\D) \subset L^2_{\beta} (\D) \to H_{1,\ab^{-1}} (\D) $ are  the identification operators that map $p \in \dom (\Ic_j) \subseteq H^1 (\D)$, $j=0,1$, to the corresponding equivalence class $ \{p +c \one : c \in \CC\}$ in $H_{1,\ab^{-1}} (\D)$.
\end{lem}
\begin{proof}
The fact that $\Uc$ is a unitary operator follows from \eqref{e:aGrad1}. The operator $\Bc$ is symmetric since $ \A|_{\GG_{\ab,\beta}} $ is symmetric. The block-matrix expression for $\Bc$ in \eqref{e:B} can obtained by direct multiplication of operators. Formula \eqref{e:B*} follows from $\A^*|_{\GG_{\ab,\beta}} = (\A|_{\GG_{\ab,\beta}})^*$.
\end{proof}

In particular, $\Bc$ is a closed symmetric operator in $H_{1,\ab^{-1}} (\D) \oplus L^2_{\beta} (\D)$.


\begin{thm} \label{t:BM-dis} 
Let $ V \in \Hom \left( L^2 (\pa \D) ,  H^{1/2} (\pa \D) \right)$ be a certain deterministic linear homeomorphism  (as in Theorem \ref{t:AM-dis}  or Lemma \ref{l:VV*}).
An operator $\wh \Bc$ is an m-dissipative extension of the symmetric operator $\Bc$ if and only if 
there exists a contraction $K$ in $L^2 (\pa \D)$ such that $\wh \Bc $ is the restriction of 
$\Bc^*$ associated with the boundary condition 
 \begin{gather} \label{e:mdisBCQ}
(K+I_{L^2 (\pa \D)}) V^{-1} \ga_0 (p)  - (K-I_{L^2 (\pa \D)}) V^\cross \gan (\ab^{-1} \nabla u) = 0 .
\end{gather}
This equivalence establishes a 1-to-1 correspondence between 
contractions $K$ in  $L^2 (\pa \D)$ and m-dissipative extensions of $\Bc$. Moreover, $\wh \Bc = \wh \Bc^*$ if and only if $K$ is a unitary operator.
\end{thm}

The proof of Theorem \ref{t:BM-dis} is given Section \ref{s:ReducingA}. This  proof, roughly speaking, states that the passage form the symmetric acoustic operator $\A$ to its part $ \A|_{\GG_{\ab,\beta}} $ does not diminish the diversity of m-dissipative extensions.

\begin{rem} \label{r:Q}
In the case of a deterministic contraction $K$,
Theorem \ref{t:BM-dis} defines the 2nd order  m-dissipative acoustic operator $\wh \Bc$ associated with  the boundary condition \eqref{e:mdisBCQ} (and with the differential operation $\bfr_{\ab,\beta}$). Lemma \ref{l:U} implies that $\wh \Bc$ is unitary equivalent to the part 
$\wh \A|_{\GG_{\ab,\beta}} $ of the 1st order acoustic operator $ \wh \A$ associated with the boundary condition \eqref{e:mdisBC}. Hence $\wh \Bc$ and $\wh \A|_{\GG_{\ab,\beta}} $ have the same spectral properties. Thus, $\wh \Bc$ has a compact resolvent if and only if 
$K+I \in \Sf_\infty$. If $K$ is a random contraction,  $\wh \Bc $  is a random-m-dissipative operator in $H_{1,\ab^{-1}} (\D) \oplus L^2_{\beta} (\D)$ and the probability that $\wh \Bc$ has a compact resolvent equals $\PP \{ K+I \in \Sf_\infty \}$.
\end{rem}

\subsection{Random impedance operators and Weyl's law for $\De^\paD$}
\label{s:RImp}

In this subsection we consider a particularly convenient randomization of generalized im\-pe\-dance boundary conditions. 
We start from several general statements.

\begin{rem}\label{r:Qimp}
Similarly to Remark \ref{r:Q}, we can adapt generalized impedance boundary conditions to the 2nd order operators $\wh \Bc$. 
Namely, let $\wh \A$ be a 1st order acoustic operator of Section \ref{s:DisBC} associated with  
a generalized impedance boundary condition $\Zc \ga_0 (p) = \gan (\vbf)$,
where $\Zc:\dom \Zc \subseteq H^{1/2} (\pa \D) \to H^{-1/2} (\pa \D)$ is an impedance operator. 
Let $\wh \Bc$ be the 2nd order acoustic operator associated with 
the modified boundary condition
\begin{gather} \label{e:QImp}
\Zc \ga_0 (p) = - \gan (\ab^{-1} \nabla u).
\end{gather}
Analogously to Definition \ref{d:bc}, this means that $\wh \Bc$ is the restriction of the operator $\Bc^*$ of \eqref{e:B*} to $\dom \wh \Bc = \{\{u,p\} \in \dom \Bc^* : \Zc \ga_0 (p) = - \gan (\ab^{-1} \nabla u)\}$.
Then  Lemma \ref{l:U} and the decompositions \eqref{e:La}-\eqref{e:LabDec} imply that 
$\wh \A|_{\GG_{\ab,\beta}} $ and $\wh \Bc$ are unitary equivalent.
Hence the results for deterministic or stochastic operators $\wh \Bc$ imply similar results for $\wh \A|_{\GG_{\ab,\beta}} $, and vice versa.
\end{rem}

\begin{rem} \label{r:Qimp2} 
In particular, an acoustic operator  $\wh \Bc$ associated with a certain impedance operator $\Zc:\dom \Zc \subseteq H^{1/2} (\pa \D) \to H^{-1/2} (\pa \D)$ via a boundary condition \eqref{e:QImp} is m-dissipative if and only if  $\Zc$  is closed and maximal accretive. Moreover, $\wh \Bc = (\wh \Bc)^*$ if and only if $\Zc^\cross = - \Zc^\cross$. These statements follow from Remark \ref{r:Qimp}, Lemma \ref{l:AcReducing}, and Theorem \ref{t:Imp}.
\end{rem}

The distribution of the random contraction  $K$ in Example \ref{ex:normal} depends on the choice of the deterministic basis $\{e_j\}_{j \in \NN}$. Bases connected with the Laplace-Beltrami operator $\De^\paD$ are especially convenient if they are combined with the homeomorphism $V$ of Lemma \ref{l:VV*}. 


Let $\{y_j\}_{j \in \NN}$ be an orthonormal basis in $L^2 (\pa \D)$ composed of eigenfunctions of the nonnegative Laplace-Beltrami operator $\De^\paD$ in such a way that $\De^\paD y_j =  \mu_j y_j $ and the corresponding eigenvalues $\{\mu_j\}_{j \in \NN}$ form a non-decreasing sequence. Note that $0 = \mu_1 \le \mu_j $ for all $j$, $\mu_j \to +\infty$ as $j\to +\infty$, and 
$\{(1+\mu_j)^{s/2} y_j\}_{j \in \NN}$ is a Riesz basis in $H^s (\pa \D)$ for every $s \in [-1,1]$,
see \cite{GMMM11}.

Let $\zeta=\{\zeta_j\}_{j \in \NN}$ be a sequence consisting of complex random variables $\zeta_j$. 
On the set $\dom \Zc^\fin$ of all finite linear combinations $\sum c_j y_j$ with coefficients $c_j \in \CC$, we define  a randomized operator $\Zc^\fin:\om \mapsto \Zc^\fin_\om$, 
\ $\Zc^\fin_\om:\dom \Zc^\fin \subset H^{1/2} (\pa \D) \to H^{-1/2} (\pa \D)$, by 
\begin{gather}
\text{$\Zc_\fin \sum c_j y_j = \sum \zeta_j c_j  y_j  $,   \qquad 
and put $\Zc^{\diag(\zeta)} := \overline{\Zc^\fin}$},
\label{e:Zfin}
\end{gather}
where $\overline{\Zc^\fin_\om}$  is the closure of a (possibly unbounded) operator 
$\Zc^\fin_\om$ as an operator from $H^{1/2} (\pa \D)$ to $ H^{-1/2} (\pa \D)$. Roughly speaking the a.s. closed operator $\Zc^{\diag(\zeta)}$ 
has  a diagonal matrix with the random entries $\zeta_j$ in the orthogonal basis $\{y_j\}_{j\in \NN}$.
Note that its representatives $\Zc^{\diag(\zeta)}_\om$ are (possibly unbounded) operators from $H^{1/2} (\pa \D)$ to  $ H^{-1/2} (\pa \D)$, and that in each of these spaces  the basis $\{y_j\}_{j\in \NN}$ is  orthogonal, but is not orthonormal.

Let $\wh \Bc^{\diag(\zeta)}$ be the 2nd order  acoustic operator associated with the modified boundary condition  
$\Zc^{\diag(\zeta)} \ga_0 (p) = - \gan (\ab^{-1} \nabla u)$.
It is easy to see from Theorem \ref{t:Imp} and Remark \ref{r:Qimp2}, that $\wh \Bc^{\diag(\zeta)}$  is a random-m-dissipative operator if and only if
\begin{gather} \label{e:diagMdis}
\text{ $\zeta_j \in  \overline{\CC}_\rr$ a.s. for all $j \in \NN$, where $ \overline{\CC}_\rr := \{z \in \CC : \re z \ge 0\}$.}
\end{gather}
 Moreover, $\wh \Bc^{\diag(\zeta)}$ is a 
random-selfadjoint operator if and only if $\re \zeta_j =0$ a.s. for all $j \in \NN$.

\begin{thm}[diagonal impedance operators] \label{t:Diag}
Assume that \eqref{e:diagMdis} holds true. Consider the random-m-dissipative acoustic operator $\wh \Bc^{\diag(\zeta)}$ associated with the boundary condition $\Zc^{\diag(\zeta)} \ga_0 (p) = - \gan (\ab^{-1} \nabla u)$. Then:
\item[(i)] The probability that $\wh \Bc^{\diag(\zeta)}$ has a compact resolvent equals 
$\PP \{\lim_{j\to \infty} \frac{\zeta_j}{\sqrt{\mu_j}} = 0\}$ (note that $\mu_j = 0$ only for $ 1 \le j \le b_0 (\pa \D) $, where $b_0 (\pa \D)$ is the number of connected components of $\pa \D$). 

\item[(ii)] If $\re \zeta_j = 0$ a.e. for all $j \in \NN$, then  $\PP \{ \si (\wh \Bc) = \si_\disc (\wh \Bc)\}  =  \PP \{\lim_{j\to \infty} \frac{\zeta_j}{\sqrt{\mu_j}} = 0\}$.

\item[(iii)] Assume additionally that random variables  $|\zeta_j|$, $j \in \NN$, are independent. Then the probability that  $\wh \Bc^{\diag(\zeta)}$ has a compact resolvent is either 0, or 1.
Moreover, $\wh \Bc^{\diag(\zeta)}$ has a compact resolvent with probability 1
if and only if, for every $\de>0$, there exists $n(\de) \in \NN$ such that 
$
0<\prod_{j = n(\de)}^{\infty} \PP \left\{ |\zeta_j|  < \de \sqrt{\mu_j}\right\} 
$.
\end{thm}
\begin{proof}
Let $\wh \Ac $
be the acoustic operator associated with the randomized boundary condition  $\Zc^{\diag(\zeta)} \ga_0 (p) = \gan (\vbf)$.
By Remark \ref{r:Qimp}, it is enough to prove the theorem  for the operator $\wh \A|_{\GG_{\ab,\beta}}$ instead of $\wh \Bc^{\diag(\zeta)}$. For $\wh \A|_{\GG_{\ab,\beta}}$, we can use \eqref{e:Pto1}.
However, we have to pass to the settings of  Example \ref{ex:normal}, which depend on the choice of  the homeomorphisms $V$ and $V^\cross$.

In order to express $\xi_j$ of of  Example \ref{ex:normal} via $\zeta_j$, it is convenient to choose $V$ and $V^\cross$ as in Lemma \ref{l:VV*}. 
In particular, $V y_j = V^\cross y_j = (1+\mu_j)^{-1/4} y_j$ for all $j \in \NN$. The boundary condition 
$\Zc^{\diag(\zeta)} \ga_0 (p) = \gan (\vbf)$
can be equivalently reformulated as \eqref{e:mdisBCR}
with the random contraction $K$ expressed as in Example \ref{ex:normal} via random variables $\xi_j$ given by 
\[
\xi_j := \frac{\zeta_j - (1+\mu_j)^{1/2}}{\zeta_j + (1+\mu_j)^{1/2}} = -1 + \frac{2 \zeta_j}{\zeta_j + (1+\mu_j)^{1/2}}  , \qquad j \in \NN.
\]
In order to obtain statement (i) from \eqref{e:Pto1}, it is enough to  notice that 
 $\lim \mu_j =  +\infty$  implies that the events $\{\om: \lim_{j \to \infty} (\xi_j+1) = 0 \}$ and  $\{\om: \lim_{j \to \infty} \frac{\zeta_j}{(1+\mu_j)^{1/2}} =0 \}$ coincide. Statement (ii) follows from (i) and the combination of Remarks \ref{r:Qimp2}, \ref{r:RandSA}, and \ref{r:DiscDet}.
Statement (iii) can be obtained by the arguments of the proof of Corollary \ref{c:normal}.
 \end{proof}

In the case where $|\zeta_j|$, $j \in \NN$, are i.i.d. random variables, Theorem \ref{t:Diag} (ii) leads to a resolvent compactness criterion written in terms of the counting function 
\[
\textstyle \Nc_{\De^{\pa \D}} (\la) := \# \{ j \in \NN : \mu_j \le \la\} = \sum_{\mu_j \le \la} 1 , \qquad \la \in \RR,
\]
for the eigenvalues of the Laplace-Beltrami operator $\De^{\pa \D}$.

\begin{thm} \label{t:iid}
Suppose \eqref{e:diagMdis} and assume additionally that $|\zeta_j|$, $j \in \NN$, are i.i.d. random variables. Then $\wh \Bc^{\diag(\zeta)} $ has a.s. a compact resolvent
if and only if the random variables $\Nc_{\De^{\pa \D}} (|\zeta_1|^2/\de^2)$ have expectations  
$ \
\EE \ \Nc_{\De^{\pa \D}} \left( |\zeta_1|^2/\de^2 \right) <+\infty \quad \text{for every $\de>0$.}
$
\end{thm}
\begin{proof}
Let $F (\cdot)= F_{|\zeta_1|} (\cdot) :=\PP \{|\zeta_1| < \cdot\}$ be a distribution function of $|\zeta_1|$ (and so a distribution function for each of $|\zeta_j|$). 
Applying Theorem \ref{t:Diag} (iii), one sees that the resolvent of $\wh \Bc^{\diag(\zeta)} $ is a.s. compact if and only if
\begin{equation}
-\infty < 
\sum_{j \ge n (\de)} \ln F ( \de \sqrt{\mu_j} ) = \sum_{j \ge n (\de)} \ln (1- [1- F ( \de \sqrt{\mu_j} )]) \quad 
\text{for all $\de>0$}.
\label{e:-inf<}
\end{equation}
Since $\mu_j \to +\infty$ and $F ( \de \sqrt{\mu_j} ) \to 1-0$ as $j \to \infty$, condition \eqref{e:-inf<} can be reformulated in terms of the convergence of the series
$
\sum_{k=1}^{+\infty} [1-F(\de \sqrt{\mu_k})]  < +\infty 
$ for all $\de>0$.
From 
\[
\sum_{k=1}^{+\infty} (1-F(\de \sqrt{\mu_k})) = \sum_{k=1}^{+\infty} \int\limits_{\de\sqrt{\mu_k}-0}^{+\infty}  \dd F
= \sum_{k=1}^{+\infty} \int\limits_{\de\sqrt{\mu_k}-0}^{\de\sqrt{\mu_{k+1}}-0}  k  \dd F = 
\sum\limits_{k=1}^{+\infty} \int\limits_{\de\sqrt{\mu_k}-0}^{\de\sqrt{\mu_{k+1}}-0} \Nc_{\De^{\pa \D}}  \left( \frac{s^2}{\de^2} \right) \dd F (s), 
\]
we see that $\displaystyle \sum_{k=1}^{+\infty} (1-F(\de \sqrt{\mu_k})) = \int_{\de\sqrt{\mu_1}-0}^{+\infty} \Nc_{\De^{\pa \D}}  \left( s^2/\de^2 \right) \ \dd F (s)$. Thus, the resolvent of $\wh \A|_{\GG_{\ab,\beta}}$ is a.s. compact
if and only if  
$
\int_\Om \Nc_{\De^{\pa \D}} (|\zeta_1|^2/\de^2) \dd \PP <+\infty$
 for every $\de>0$. 
\end{proof}

\begin{rem} \label{r:bounded}
If the distribution of i.i.d. random variables $|\zeta_j|$ has a bounded support in Theorem \ref{t:iid}, then one can easily see that the resolvent of $\wh \Bc^{\diag(\zeta)}$ is a.s. compact. 
\end{rem}
\begin{rem} \label{r:zjiR}
 If $\zeta_j$ are i.i.d. random variables with a distribution supported on the purely imaginary line $\ii \RR := \{\ii c : c \in \RR\}$ and satisfying  
$ 
\EE \ \Nc_{\De^{\pa \D}} \left( |\zeta_1|^2/\de^2 \right) =\infty$
for a certain $\de>0$, then $\si_\ess (\wh \Bc^{\diag(\zeta)}) \neq \emptyset$ with probability 1.
This follows from  Theorems \ref{t:iid} and \ref{t:Diag}.
\end{rem}

The condition of Theorem \ref{t:iid} can be simplified if additional properties of the counting function $\Nc_{\De^{\pa \D}} $ are known, e.g.,
if it is known that  
\begin{equation} \label{e:asymp}
 \Nc_{\De^{\pa \D}} (\la) \asymp  \la^{(d-1)/2} \quad \text{ as $\la \to +\infty$}.
\end{equation}
Here $g \asymp h$ means that there exists constants $c_1$ and $c_2$ such that $0<c_1 \le g(\la)/h(\la) \le c_2$ for sufficiently large $\la$. Recall that $\D$ is a Lipschitz domain in $\RR^d$. 

\begin{cor}\label{c:Weyl}
Assume that $\zeta_j$, $j \in \NN$, are i.i.d. $\overline{\CC}_\rr$-valued random variables. Assume that asymptotic estimate \eqref{e:asymp} is satisfied for the Laplace-Beltrami operator $\De^{\pa \D}$ on the boundary $\pa \D$. 
Then $\wh \Bc^{\diag(\zeta)}$ has a.s. a compact resolvent if and only if 
the i.i.d. random variables $|\zeta_j|$ have a finite $(d-1)$-th raw moments $\EE |\zeta_1|^{(d-1)} <+\infty$.
\end{cor}

\begin{proof}
The statements follow  immediately from Theorem \ref{t:iid}. 
\end{proof}

\begin{rem} \label{r:WeylR}
Assume that, for a certain $\vep>0$, the boundary $\pa \D$ has $C^{2,\vep}$-regularity.
Then the Weyl asymptoics \cite{Z98,Z99,I00,I19} implies the asymptotic estimate \eqref{e:asymp}. Thus, Corollary \ref{c:Weyl} is applicable. In particular,
if i.i.d. random variables $\zeta_j$, $j \in \NN$, have the distribution supported on the purely imaginary line $\ii \RR $ and $\EE |\zeta_1|^{(d-1)} = \infty$, then a.s. $\si_\ess (\wh \Bc^{\diag(\zeta)}) \neq \emptyset$.
\end{rem}

The case of the deterministic constants $\xi_j = \ii c$ with $c >0$ corresponds to a model of a 2-D photonic crystal surrounded by a superconductor  \cite{LL84,FK96,ACL18,EK22}.  The following example can be seen as a randomization of this model.

\begin{ex}[Pareto randomized superconducting boundary]\label{ex:Pareto}
Assume that the boundary \linebreak $\pa \D$ has $C^{2,\vep}$-regularity for a certain $\vep>0$.
Assume that purely imaginary random variables $\ii \zeta_j $, $j \in \NN$, are i.i.d. with the Pareto distribution 
$\mathrm{Pa}(a,s_{\min})$ (this means that they are real-valued and have the density function 
$f_{\ii \zeta_1} (s)$, $s \in \RR$, that is equal to $0$ for $s \le s_{\min}$ and is equal to 
$a s_{\min}^a/s^{a+1}$ for $s>s_{\min}$, where the parameters $a$ and $s_{\min}$ are positive constants). Then Remark \ref{r:WeylR} implies that $\si (\wh \Bc^{\diag(\zeta)}) = \si_\disc (\wh \Bc^{\diag(\zeta)})$ a.s. if and only if $d<a+1$. In other words, the critical value $a_{\mathrm{c}} = d-1$ corresponds to  the emergence of an essential spectrum with probability 1  for the acoustic operator $\wh \Bc^{\diag(\zeta)}$.
\end{ex}

\section{Boundary tuples and the proofs of Theorems \ref{t:AM-dis} and  \ref{t:Imp}}
\label{s:BT+proof}

\subsection{M-boundary tuples of abstract trace spaces and trace maps}
\label{s:MBT}


The triple $(\Hs_{-,+},\Hs,\Hs_{+,-})$ consisting of Hilbert spaces is called a \emph{generalized rigged Hilbert space} 
if spaces $\Hs_{\mp,\pm}$ are dual to each other w.r.t. the $\Hs$-pairing $\<\cdot,\cdot\>_\Hs$ that is obtained from the inner product $(\cdot|\cdot)_\Hs$ of the \emph{pivot space} $\Hs$ as an extension by continuity to 
$\Hs_{\mp,\pm} \times \Hs_{\pm,\mp}$, see  \cite[Vol.2]{RS},  \cite{S21,EK22}, and also Definition \ref{d:GRHSp} in Appendix \ref{ap:A}. This notion is convenient for the description of duality of trace spaces for various wave equations \cite{S21,EK22} and does not assume any of embeddings between the spaces $\Hs_{\mp,\pm}$ and $\Hs$. 

A generalized rigged Hilbert space $(\Hs_+,\Hs,\Hs_-)$ with the continuous embeddings \linebreak $\Hs_+ \imb \Hs \imb \Hs_-$ is a rigged Hilbert space (Gelfand triple) in the standard sense. 
An example of a generalized rigged Hilbert spaces that is not a rigged Hilbert space can be produced by the inversion of the order of spaces in a rigged Hilbert space, e.g., $(H^{-1/2} (\pa \D), L^2 (\pa \D), H^{1/2} (\pa \D))$.
Less trivial examples without any embeddings between $\Hs_{\mp,\pm}$ and $\Hs$ appear in the context of trace spaces for Maxwell operators \cite{EK22}.
On an abstract level, the emergence of the trace spaces from the integration by parts is described by the following definition.

\begin{defn}[m-boundary tuple, \cite{EK22}] \label{d:MBT}
Let $\Ao$ be a closed densely defined symmetric operator in a Hilbert space $\Xf$.
Let auxiliary Hilbert spaces $\Hs$ and $\Hs_{\mp,\pm}$ form  a generalized rigged Hilbert space $(\Hs_{-,+},\Hs,\Hs_{+,-})$. 
We say that $(\Hs_{-,+},\Hs, \Hs_{+,-}, \Ga_0,\Ga_1)$ is an \emph{m-boundary tuple}  for the adjoint operator $\Ao^*$
if the following conditions hold:
\item[(M1)] the map 
$\Ga : f \mapsto \{ \Ga_0 f , \Ga_1 f \} $ is a surjective linear operator from $\dom \Ao^*$ onto $\Hs_{-,+} \oplus \Hs_{+,-}$; 
\item[(M2)] $
(\Ao^*f |g)_{\Xf} - (f |\Ao^*g)_{\Xf} =  \< \Ga_1 f , \Ga_0 g \>_\Hs  - \<\Ga_0 f , \Ga_1 g \>_\Hs $ for all $f,g \in \dom \Ao^*$.
\end{defn}

\begin{rem}  \label{r:RBT}
If additionally the continuous embeddings $\Hs_{-,+} \imb \Hs \imb \Hs_{+,-}$ of a rigged Hilbert space take place in Definition \ref{d:MBT}, we call   
$(\Hs_{-,+},\Hs,\Hs_{+,-},\Ga_0,\Ga_1)$ a \emph{rigged boundary tuple}  for  $\Ao^*$.
Somewhat similar constructions appeared in a less explicit form in \cite{DM95,BM14}. 
\end{rem}
\begin{rem} \label{r:BT}
If $(\Hs, \|\cdot \|_\Hs) = (\Hs_{-,+}, \| \cdot \|_{\Hs_{-,+}}) = (\Hs_{+,-}, \| \cdot \|_{\Hs_{+,-}})$, i.e., if the duality is trivial,  an m-boundary tuple $(\Hs_{-,+},\Hs,\Hs_{+,-},\Ga_0,\Ga_1)$ becomes effectively a \emph{boundary triple} $(\Hs,\Ga_0,\Ga_1)$ in the standard sense introduced by Talyush, Kochubei, and Bruk (see \cite{K75,GG91,P12,DM17,BHdS20}). 
\end{rem}

Boundary triples usually do not fit exactly to the specifics of the integration by parts for PDEs. Therefore various regularizations and generalizations for boundary triples were suggested, see \cite{GG91,DM95,P12,AGW14,BM14,KZ15,DM17,DHM20,S21,DHM22,BGM22,EK22} and references therein.  
The notion of m-boundary tuple was introduced in \cite{EK22} in order to handle trace spaces for the Maxwell system, but is suitable also for other types of wave equations. In particular, m-boundary tuples allows us to   connect  various dual objects, see Sections \ref{s:IntParts} and \ref{s:DtN}.

The operators $\Ga_0$ and $\Ga_1$ in (M1)-(M2) are abstract analogues of traces (boundary maps). The existence of an m-boundary tuple for $\Ao^*$ is equivalent to the existence of a boundary triple for $\Ao^*$ \cite{EK22} (see Lemma \ref{l:BT} below), and so is equivalent to the statement that $\Ao$ has equal deficiency indices.

In the sequel, we assume that   
\begin{gather}
\label{e:BT1} \text{$\Ao$ is a closed symmetric densely defined  operator  in a Hilbert space $\Xf$,}\\
\label{e:BT2} \text{and $(\Hs_{-,+},\Hs, \Hs_{+,-}, \Ga_0,\Ga_1)$ is an m-boundary tuple for $\Ao^*$.}
\end{gather}

 Similarly to the theory of boundary triples \cite{GG91}, it is easy to see  that 
\begin{equation} \label{e:domA}
\dom \Ao = \{\psi \in \dom \Ao^* \ : \ 0 = \Ga_0 \psi = \Ga_1 \psi\}
\end{equation}
and  that the operators $\Ga_0 : \dom \Ao^* \to \Hs_{-,+}$ and $\Ga_1 : \dom \Ao^* \to \Hs_{+,-}$ are bounded, where 
$\dom \Ao^*$ is perceived as a Hilbert space with the graph norm \cite{EK22}.

\subsection{Extension theory and deterministic m-dissipative boundary conditions}
\label{s:Ext}

A \emph{linear relation} from a Hilbert space $\Hs_1$ to a Hilbert space $\Hs_2$ is by definition a linear subspace in the orthogonal sum $\Hs_1 \oplus \Hs_2$. A graph $\Gr B$ of an operator $B:\dom B \subseteq \Hs_1 \to \Hs_2$ is an example of a linear relation. Identifying operators with their graphs, one considers operators as particular cases of linear relations.
Then the notions of closed relation and closed operator are consistent (see \cite{K75,DM95,DM17,BHdS20} for the calculus of linear relations). If $\Th$ is a linear relation from $\Hs$ to $\Hs$, $\Th$ is said to be a linear relation in $\Hs$.

We assume in this subsection that assumptions \eqref{e:BT1}-\eqref{e:BT2} are fulfilled for a symmetric operator $\Ao$. 
An operator $\wt \Ao$ is called an \emph{intermediate extension of} $\Ao$ if 
$ \Gr \Ao \subseteq \Gr \wt \Ao \subseteq \Gr \Ao^*.$

By \eqref{e:domA},  $\dom \Ao = \ker \Ga$. There exists 1--to--1 correspondence between: 
(i) the family of intermediate extensions $\wt \Ao$ of $\Ao$, (ii) the family of subspaces $\dom (\wt \Ao) /\dom (\Ao)$ of the factor space $\dom (\Ao^*) / \dom (\Ao)$, and (iii)
the family of linear relations $\Theta$ from $\Hs_{-,+}$ to $\Hs_{+,-}$.
In the correspondence (i) $\leftrightarrow$ (iii), the linear relation $\Th = \left\{ \{ \Ga_0 u , \Ga_1 u \} \ : \ u \in \dom \wt \Ao \right\}$ is associated with the restriction $\wt \Ao = \Ao_\Theta$ of $\Ao^*$ defined by
\begin{gather} \label{e:ATh}
\text{$\Ao_\Theta := \Ao^*\uph_{\dom \A_\Th}$, \qquad  $\dom \Ao_\Th := \Ga^{-1} \Th = \left\{ u \in \dom \Ao^* \ : \ \{ \Ga_0 u , \Ga_1 u \} \in \Th \right\}$}.
\end{gather}

Let $\Th$ be a linear relation from $\Hs_{-,+}$ to $\Hs_{+,-}$.
A numerical cone of $\Th$ is defined by 
$
\Ncone (\Th) := \{ \< h_{-,+}| h_{+,-} \>_\Hs \ : \ \{ h_{-,+} , h_{+,-} \} \in \Th\}  .
$
A linear relation $\Th$ is said to be symmetric, nonnegative, dissipative,  or accretive if $\Ncone (\Th) \subseteq \RR$, $\Ncone (\Th) \subseteq [0,+\infty) $, $\Ncone (\Th) \subseteq \overline{\CC}_- $, or $\Ncone (\Th) \subseteq \overline{\CC}_\rr $, respectively. 
Besides, a linear relation $\Th$ in each of these classes is called maximal if it cannot be extended to another linear relation of the same class.

The linear relation $\Th^\cross$ is called an $\Hs$-pairing-adjoint (or, in short, $\cross$-adjoint) to $\Th$ if   
$\Th^\cross$  consist of all $\{g_{-,+},g_{+,-}\} \in \Hs_{-,+} \oplus \Hs_{+,-}$ such that $\<h_{+,-} | g _{-,+}\>_\Hs  =  \<h_{-,+} | g_{+,-}\>_\Hs$ for all  $\{ h_{-,+} , h_{+,-} \} \in \Th$.
A linear relation $\Th$ is called $\cross$-selfadjoint if $\Th =\Th^\cross$. 
Note that $\Th^\cross$ is a closed linear relation from $\Hs_{-,+}$ to $\Hs_{+,-}$ and that $(\Th^\cross)^\cross$ is the closure $\overline{\Th}$ of $\Th$ in $\Hs_{-,+} \oplus \Hs_{+,-}$. 
In the case of the trivial duality $\Hs_{\mp,\pm} = \Hs$, the above definitions are standard for linear relations in $\Hs$. In particular, $\Th^\cross$ is a generalization of the adjoint linear relation $\Th^*$ in $\Hs$. 

A linear operator $T:\dom T \subseteq \Hs_{-,+} \to \Hs_{+,-}$ is called symmetric, nonnegative, dissipative,   or accretive if the linear relation $\Gr T$ is so. If $\dom T$ is dense in $\Hs_{-,+}$, the $\cross$-adjoint linear operator $T^\cross$ can be defined via the equality $\Gr (T^\cross) = (\Gr T)^\cross$.
In particular, for $T \in \Lc (\Hs_{-,+},\Hs_{+,-})$,  one has  $T^\cross \in \Lc (\Hs_{-,+},\Hs_{+,-})$.
This allows one to define the real and imaginary parts of such $T$ as 
$\re T := \frac12 (T+T^\cross)$  and $ \im T := \frac1{2\ii}(T-T^\cross) $.

\begin{rem} \label{r:HsAdjoint}
The above definitions of  $\Hs$-pairing-adjoints for operators and linear relations can be naturally extended \cite{EK22} to operators and linear relations acting between spaces $\Hs$ and $\Hs_{\mp,\pm}$. For example, for a bounded operator $W: \Hs \to \Hs_{\mp,\pm}$, there exists a unique bounded operator $W^\cross: \Hs_{\pm,\mp} \to \Hs$ such that 
$\< W f|g\>_\Hs = (  f|W^\cross g)_\Hs$   for all $ f \in \Hs$ and $g \in \Hs_{\pm,\mp}$.
This definition of the $\cross$-adjoint is used in Section \ref{s:DisBC}.
Note that $(W^\cross)^\cross = W$. For $W \in \Hom (\Hs,\Hs_{\mp,\pm})$, one has $W^\cross \in \Hom (\Hs_{\pm,\mp}, \Hs)$ and $(W^\cross)^{-1} = (W^{-1})^\cross$.
\end{rem}


For $\Ao$ satisfying \eqref{e:BT1}-\eqref{e:BT2}, Definition \ref{d:MBT} implies that 
$\Gr \Ao_{\Th_1} \subseteq  \Gr \Ao_{\Th_2}$ if and only if $\Th_1 \subseteq \Th_2$.
Besides, $(\Ao_\Th)^* = \Ao_{\Th^\cross}$ and, for the closures, one has $\overline{\Ao_\Theta} = \Ao_{\overline{\Th}}$. 
Combining these facts with the property (M2) of Definition \ref{d:MBT}, one can easily show \cite{EK22} that 
 \begin{multline} \label{e:ATheqv}
\text{$\Ao_\Theta$ is closed, selfadjoint, (maximal) dissipative, symmetric, accretive,}
\\\text{  if and only if $\Th$ is so as a linear relation from $\Hs_{-,+}$ to $\Hs_{+,-}$.}
\end{multline} 

\begin{rem} \label{r:AThM-dis}
Any maximal dissipative linear relation is closed. Using \eqref{e:CrM-dis1}, we see that
  \begin{equation} \label{e:ATheqvMdis}
\text{$\Ao_\Theta$ is m-dissipative if and only if  $\Th$ is maximal dissipative.}
\end{equation} 
Every dissipative extension $\wt \Ao$ of $\Ao$ is an intermediate extension \cite{P59,GG91}
and can be represented as $\wt \Ao = \Ao_\Theta$. Thus, 
\eqref{e:ATheqvMdis} describes all m-dissipative extensions of $\Ao$, and simultaneously all m-dissipative restrictions of $\Ao^*$. The following result writes this description via `abstract boundary conditions'
extending the idea of Kochubei \cite{K75} to m-boundary tuples.
\end{rem}

Let us fix certain linear homeomorphism $W \in \Hom (\Hs, \Hs_{-,+})$.
\begin{prop}[\cite{EK22}] \label{p:absM-dis} 
An operator $\wh \Ao$ is an m-dissipative extension of $\Ao$ if and only if, for a certain contraction $K: \Hs \to \Hs$, the operator  $\wh \Ao$ is the restriction of $\Ao^*$  associated with the `abstract boundary condition' 
 \begin{gather} \label{e:K+IGa}
(K+I_\Hs) W^{-1} \Ga_0 f + \ii (K-I_\Hs) W^\cross  \Ga_1 f = 0 .
\end{gather}
This equivalence establishes a bijective correspondence between m-dissipative extensions of $\Ao$ and contractions $K$ in $\Hs$. Besides, $\wh \Ao = \wh \Ao^*$ if and only if $K$ is unitary.
\end{prop}

\begin{prop} \label{p:R-RinS}
In the settings of Proposition \ref{p:absM-dis}, let us consider two contractions $K_j:\Hs \to \Hs$, $j=1,2$,
and two corresponding m-dissipative restrictions $\wt \Ao_j$ of $\Ao^*$ defined by the abstract boundary conditions \eqref{e:K+IGa} with $K=K_j$, $j=1,2$.
Let $\la \in \rho (\wt \Ao_1 ) \cap \rho (\wt \Ao_2 ) $ and $1 \le p \le \infty$. 
Then
$(\wt \Ao_2 - \la)^{-1} - (\wt \Ao_1 - \la)^{-1}  \in \Sf_p (\Xf)$ if and only if $K_2 - K_1 \in \Sf_p (\Hs)$.
\end{prop}

\begin{proof}
It is proved in \cite{EK22} that $(\Hs, W^{-1} \Ga_0 , W^\cross  \Ga_1)$ is a boundary triple for $\Ao^*$ (see Lemma \ref{l:BT} below). Therefore, the direct application of \cite[Theorem 3.1]{GG91} proves the proposition.
\end{proof}

\subsection{Acoustic boundary tuples and the proof of Theorems \ref{t:AM-dis} and \ref{t:Imp}\label{s:IntParts}}

The  space $L^2 (\pa \D)$ will play the role of a pivot space. Its scalar product $(\cdot|\cdot)_{L^2 (\pa \D)} $ generates for each $s \in (0,1]$ two sesquilinear forms that provide the parings of the Sobolev space $H^{-s} (\pa \D)$ with $H^{s} (\pa \D) $,  and conversely, of $H^{s} (\pa \D)$ with $H^{-s} (\pa \D) $. These two sesquilinear forms are denoted in the same way $\<\cdot,\star\>_{L^2 (\pa \D)}$ and are called the $L^2 (\pa \D)$-pairings.  We use these $L^2 (\pa \D)$-pairings for the rigged Hilbert space $H^{1/2} (\pa \D) \imb  L^2 (\pa \D) \imb H^{-1/2} (\pa \D)$ and the generalized rigged Hilbert space 
$(H^{-1/2} (\pa \D),  L^2 (\pa \D) ,  H^{1/2} (\pa \D))$ (see Section \ref{s:MBT} and Appendix \ref{ap:A}).

The integration by parts formula
$
\int_\D \vbf \cdot \nabla p + \int_\D p (\nabla \cdot \vbf) = \int_{\pa \D} p (\vbf \cdot \n) ,
$
which is valid for regular enough $\vbf:\D\to \CC^d$ and $p:\D \to \CC$,  can be extended \cite{M03,ACL18} to the  equality 
\begin{gather} \label{e:IbyP}
( \gradm p | \vbf)_{L^2 (\D,\CC^3)} + (p | \Div \vbf)_{L^2 (\D)} = \<\ga_0 p | \ga_n \vbf \>_{L^2 (\pa \D)} , \quad p \in H^1 (\D), \ \vbf \in \HH (\Div, \D),
\end{gather}
where $\ga_0  \in \Lc (H^1 (\D), H^{1/2} (\pa \D))$  and $ \gan \in \Lc (H^1 (\D,\CC^d), H^{-1/2} (\pa \D) )$
are the surjective operators of scalar and normal trace, respectively (see Section \ref{s:DisBC} and \cite{M03,ACL18}).
The integration by parts for the operator $\A^*$ of  Section \ref{s:DisBC} can be written via m-boundary tuples.

\begin{prop} \label{p:acuBT}
Let us define  $\wh \ga_0  (\{\vbf,p\}) := \ga_0 (p) $ and $\wh \gan (\{\vbf,p\}) := \gan (\vbf)$. Then
\begin{multline} \label{e:mTf}
 \Tc :=(H^{1/2} (\pa \D), L^2 (\pa \D) , H^{-1/2} (\pa \D),  \wh \ga_0, \ (-\ii) \wh \gan ) 
\\ \text{ and } \qquad \Tc_* :=(H^{-1/2} (\pa \D), L^2 (\pa \D) , H^{1/2} (\pa \D),  \wh \gan, (- \ii)\wh \ga_0  ) 
\end{multline} 
are m-boundary tuples for the acoustic operator $\A^*$. We say that these two m-boundary tuples are dual to each other.
\end{prop}

\begin{proof}
The statement follows from the integration by parts \eqref{e:IbyP}, the mutual duality of the spaces $H^{\pm 1/2} (\pa \D)$ w.r.t. $L^2 (\pa \D) $, and the surjectivity of $\ga_0$ and $ \gan$.
\end{proof}

\begin{proof}[Proof of Theorem \ref{t:AM-dis}]
The combination  of Propositions \ref{p:absM-dis} and \ref{p:acuBT} for the m-boundary tuple $\Tc$ immediately proves Theorem \ref{t:AM-dis}.
\end{proof}

\begin{proof}[Proof of Theorem \ref{t:Imp}]
Theorem \ref{t:Imp} can be proved similarly to the analogous result for the Maxwell system in \cite[Corollary 2.3]{EK22}. That is, the abstract result of  \cite[Corollary 7.2]{EK22} is applied to 
the m-boundary tuple $\Tc$ for the operator $\Ac^*$. 
\end{proof}

\section{Dirichlet-to-Neumann maps and Krein resolvent formulae}
\label{s:DtNKrF}

In the settings of boundary triples, abstract Weyl functions and generalized Krein resolvent formulae were  developed by Derkach and Malamud \cite{DM91,DM95} (see also the monographs \cite{DM17,BHdS20}). In this section, a part of this theory is extended  to the case of m-boundary tuples and, using Proposition \ref{p:acuBT},  applied to acoustic operators in order to prove Theorems \ref{t:NtD} and \ref{t:aKrF}.

In this section, we suppose that a symmetric operator $\Ao$ satisfies assumptions  
\eqref{e:BT1}--\eqref{e:BT2} of Section \ref{s:MBT} and that a certain linear homeomorphism $W \in \Hom (\Hs,\Hs_{-,+})$ is chosen.

\subsection{Neumann-to-Dirichlet maps and the proof of Theorem \ref{t:NtD}}
\label{s:DtN}

Fixing an arbitrary $W \in \Hom (\Hs,\Hs_{-,+})$, one can construct boundary triples from a boundary tuple with the use of the following lemma (for the definition of a boundary triple, see Remark \ref{r:BT} and, in more details, the monographs \cite{GG91,P12,DM17,BHdS20}).

\begin{lem}[\cite{EK22}]\label{l:BT}
 Let $(\Hs_{-,+},\Hs,\Hs_{+,-}\Ga_0,\Ga_1)$ be an m-boundary tuple (for $\Ao^*$). Then
 \[
\text{$ \Tf^W = (\Hs, \Ga_0^W, \Ga_1^W) := (\Hs, W^{-1} \Ga_0, W^\cross \Ga_1) $ 
\quad and \quad $\Tf^W_* = (\Hs, \ii \Ga_1^W,  (-\ii) \Ga_0^W)$}
\]
 are  boundary triples (which we call dual to each other).
\end{lem}

It follows from Definition \ref{d:MBT} of m-boundary tuple that, for $j=0,1$, the restrictions 
\begin{gather} \label{e:whA01}
\wh \Ao_j = \Ao^* \uph_{\ker \Ga_j} = \Ao^* \uph_{\ker \Ga_j^W} \text{are selfadjoint extensions of $\Ao$.} 
\end{gather}

Applying \cite[Section 1.2]{DM91} to the boundary triple $\Tf^W$, one sees that there exists
a unique holomorphic function  $M_W : \rho (\wh \Ao_0) \to \Lc (\Hs)$ such that 
$M_W (z) \Ga^W_0 f_z = \Ga^W_1 f_z $  for all $f_z \in \ker (\Ao^* - z I_\Xf)$, \quad $z  \in \rho (\wh \Ao_0)$. 
In the terminology of \cite{DM91}, $M_W$ is called  \emph{the Weyl function  corresponding to the boundary triple $\Tf^W$ (for  $\Ao^*$)}. Note that 
the Weyl function corresponding to the \emph{dual boundary triple} $\Tf^W_*$ is defined on $\rho (\wh \Ao_1)$ and equals $(-1)(M_W (z))^{-1}$ for $z \in \rho (\wh \Ao_0) \cap \rho (\wh \Ao_1)$. In particular,
 one has  $M_W (z) \in \Hom (\Hs)$ for all $z \in \rho (\wh \Ao_0) \cap \rho (\wh \Ao_1)$.

The abstract versions of DtN and NtD maps associated with the m-boundary tuple $(\Hs_{-,+},\Hs,\Hs_{+,-}\Ga_0,\Ga_1)$ now can be defined 
via 
\begin{gather} \label{e:M=MV1}
\text{\qquad $M (z) :=(W^\cross)^{-1} M_W (z  ) W^{-1} $ for $z \in \rho (\wh \Ao_0)$, } \\
\text{and  \quad
$\wt M (z) := (-1)(M(z))^{-1}$ for $z \in \rho (\wh \Ao_0) \cap \rho (\wh \Ao_1)$.
}
\label{e:M=MV2}
\end{gather}

\begin{rem} \label{r:indM}
 $\wt M $ can be extended to an $\Lc (\Hs_{+,-},\Hs_{-,+})$-valued holomorphic function in $\rho (\wh \Ao_1) $. The definitions of $M$ and $\wt M$ are independent of the choice of $W$ and lead to the following.
\end{rem}

\begin{prop}[$\Ga_0$--to--$\Ga_1$ map] \label{p:WeylF}
There exists a unique holomorphic function \linebreak
$
M(z): \rho (\wh \Ao_0) \to \Lc (\Hs_{-,+},\Hs_{+,-})  
$
that satisfies the equality 
\quad $M (z) \Ga_0 f_z = \Ga_1 f_z $ \quad for all $f_z \in \ker (\Ao^* - z)$ and $z  \in \rho (\wh \Ao_0) $. Moreover, this function 
 has the following properties:
\item[(i)] The operator $(\im z) \im M (z)$ is nonnegative for every $z \in \CC \setminus \RR$ (in the sense of Section \ref{s:Ext}).
\item[(ii)] $M (\overline{z}) = (M (z))^\cross$ for all $z \in \rho (\wh \Ao_0)$.
\item[(iii)] $M (z) \in \Hom (\Hs_{-,+},\Hs_{+,-})$  for every $z \in \rho (\wh \Ao_0) \cap \rho (\wh \Ao_1)$.\\
We say that this function $M$ is the \emph{$\Ga_0$--to--$\Ga_1$ map} \emph{corresponding to the m-boundary tuple $(\Hs_{-,+},\Hs,\Hs_{+,-}\Ga_0,\Ga_1)$} (for the operator $\Ao^*$).
\end{prop}
\begin{proof}
Since $W \in \Hom (\Hs,\Hs_{-,+})$ and since $W^\cross \in \Hom (\Hs_{+,-},\Hs)$ is the $\cross$-adjoint of $W$, the proposition follows from the combination of formulae \eqref{e:M=MV1}-\eqref{e:M=MV2}, Lemma \ref{l:BT}, and Remark \ref{r:HsAdjoint} with the properties of abstract Weyl functions in \cite[Section 1]{DM91}  (or in  \cite{DM95,DM17}).   
\end{proof}

\begin{proof}[Proof of Theorem \ref{t:NtD}]
Now, we are ready to consider 
the NtD map $M_{\NtD}$ for the acoustic system. Using Proposition \ref{p:WeylF} together with the m-boundary tuple 
\[
\Tc_*  =(H^{-1/2} (\pa \D), L^2 (\pa \D) , H^{1/2} (\pa \D),  \wh \gan, (- \ii)\wh \ga_0  ) =  
(\Hs_-,\Hs,\Hs_+,\Ga_0,\Ga_1) \text{ of \eqref{e:mTf},}
\] 
we obtain the corresponding $\Ga_0$--to--$\Ga_1$ map $M_{\NtD}: \rho (\A^\Nr)\to \Lc (H^{-1/2} (\pa \D), H^{1/2} (\pa \D)) $.
The dual m-boundary tuple  
$\Tc  = (\Hs_+,\Hs,\Hs_-, \wt \Ga_0, \wt \Ga_1) =  (\Hs_+,\Hs,\Hs_-, \ii \Ga_1, (-\ii) \Ga_0)$ and Proposition \ref{p:WeylF} produce the $\wt \Ga_0$--to--$\wt \Ga_1$ map $M_{\DtN}: \rho (\A^\Dr) \to \Lc (H^{1/2} (\pa \D), H^{-1/2} (\pa \D))  $, which corresponds to $\wt M$ of \eqref{e:M=MV2} and to the DtN-map of Remark \ref{r:DtN}. Clearly, $M_{\DtN} (z) = - (M_{\NtD} (z))^{-1}$  for all $z \in  \rho (\A^\Nr) \cap \rho (\A^\Dr) $. This completes the proof of Theorem \ref{t:NtD}.
\end{proof}

\subsection{Krein resolvent formulae and the proof of Theorem \ref{t:aKrF}}
\label{s:KreinF}

Let  $\Ao_\Th$ be an intermediate extension of $\Ao$ defined by \eqref{e:ATh} with the use of a linear relation $\Th \subseteq \Hs_{-,+} \oplus \Hs_{+,-}$.
Translating \eqref{e:ATh} into the settings of the boundary triple $\Tf^W$
of Lemma \ref{l:BT}, we see that 
$\dom \Ao_\Th 
= \left\{ u \in \dom \Ao^* \ : \ \{ W^{-1} \Ga_0 u , W^\cross \Ga_1 u \} \in \Th_W \right\},$
where  
\begin{gather} \label{e:ThW}
\Th_W := 
\left\{\  \{W^{-1} h_{-,+}, W^\cross h_{+,-} \} \ : \ \{ h_{-,+},h_{+,-}\} \in \Th \ \right\}
\text{ is a linear relation in $\Hs$.}
 \end{gather}

By \cite[Proposition 1.6]{DM95}, 
for $z \in \rho (\wh \Ao_0)$, one has $ z \in \rho (\Ao_\Th) $ 
exactly when  $0 \in \rho (\Th_W - M_W (z))$,
or, equivalently,
exactly when $(\Th - \M (z))^{-1} \in \Lc (\Hs_{+,-}, \Hs_{-,+})$.
Here $\Th_W - M_W (z)$, $\Th - \M (z)$, and the resolvent set $\rho (\Th_W - M_W (z))$ are understood  in the sense of the calculus of linear relations, which also allows one to express $\rho (\Th_W - M_W (z))$ and the inverse linear relation $(\Th_W - M_W (z))^{-1}$ in terms of equalities with operator coefficients (see \cite{DM95,DM17,BHdS20}). 
We use such a transition from linear relations to equalities with operator coefficients for the generalized Krein formula \cite[formula (3.39)]{DM95}, which for the boundary triple $\Tf^W$ and $z \in \rho ( \Ao_\Th) \cap \rho (\wh \Ao_0)$ can be  written as 
\begin{align} 
( [ \Ao_\Th  - z ]^{-1} f | g)_\Xf  - & ([\wh \Ao_0 - z ]^{-1} f | g)_\Xf 
 \label{e:KFTh0}
 \\
 & \quad = 
  \left( [\Th_W -M_W (z)]^{-1} W^\cross  \Ga_1 [\wh \Ao_0-z   ]^{-1} f \mid W^\cross \Ga_1 [\wh \Ao_0 - \overline{z}]^{-1} g\right)_\Hs \label{e:KFTh1}
\\ & \quad   =  \left\< [\Th -M (z)]^{-1} \Ga_1 [\wh \Ao_0-z   ]^{-1} f \mid \Ga_1 [\wh \Ao_0 - \overline{z}]^{-1} g\right\>_\Hs  .
\label{e:KFTh2}
\end{align} 

In the particular case where $\Th_W$ is a graph $\Gr E$ of an operator $E:\dom E \subseteq \Hs \to \Hs$, one sees from \cite{DM91,DM17} that $\rho (E - M_W (z)) = \rho (\Th_W - M_W (z))$ and 
\begin{gather*} 
\text{$(\Th_W - M_W (z))^{-1} =(E - M_W (z))^{-1}$ for all $z \in \rho (\wh \Ao_0) $ such that $0 \in \rho (E - M_W (z))$.} 
\end{gather*}
In a more general case, where a closed restriction $\Ao_\Th$ of $\Ao^*$ is defined by 
an abstract boundary condition 
\begin{gather} \label{e:B0B1}
\text{$ E_0 \Ga^W_0 f + E_1 \Ga^W_1 f =0$ with    
$E_0, E_1 \in \Lc (\Hs)$ such that $0 \in \rho (E_0 E_0^* + E_1 E_1^*)$},
\end{gather}
\cite[Corollary 7.44]{DM17} implies that  
\begin{gather} \label{e:B0B1rho}
\rho (\Ao_\Th) \cap \rho (\wh \Ao_0) = \{ z \in \rho (\wh \Ao_0) : 0 \in \rho (E_0 +E_1 M_W (z)) \} \quad \text{ (see also \cite{DM95,BHdS20}).}
\end{gather}
It follows from 
\cite[Corollary 7.103]{DM17} that formula \eqref{e:KFTh0}-\eqref{e:KFTh1} 
can be written  as 
\begin{multline} \label{e:KF}
 ( [\wh \Ao_0 - z ]^{-1} f | g)_\Xf - ( [\Ao_\Th - z ]^{-1} f | g)_\Xf \\
 =  \left( [E_0 +  E_1 M_W (z)]^{-1} E_1 W^\cross \Ga_1 [\wh \Ao_0-z ]^{-1} f \mid W^\cross \Ga_1 [\wh \Ao_0- \overline{z}  ]^{-1}g\right)_\Hs .
\end{multline}

We apply now these formulae to the abstract boundary condition 
\begin{gather} \label{e:tK+IGa}
(\wt K+I_\Hs) W^{-1} \Ga_0 f + \ii (\wt K-I_\Hs) W^\cross  \Ga_1 f = 0 .
\end{gather}

\begin{prop}[abstract Krein's resolvent formula] \label{p:KrFabst}
Let $\wt K$ be  a certain contraction  in $\Hs$, and let $\wh \Ao$ be the m-dissipative restriction of $\Ao^*$ defined by \eqref{e:tK+IGa}.
Then:
\item[(i)] A complex number $z \in \rho (\wh \Ao_0) $ belongs to $ \rho (\wh \Ao)$ if and only if 
\begin{gather}
\text{$E_0 W^{-1}   + E_1 W^\cross M ( z) \in \Hom (\Hs_{-,+}, \Hs) $, where 
$ E_0 = \wt K+I_\Hs$ and $ E_1 = \ii (\wt K-I_\Hs) $.} \label{e:E0E1abs}
\end{gather}
\item[(ii)] For $z \in \rho (\wh \Ao)  \cap \rho (\wh \Ao_0) $ and arbitrary $f,g \in \Xf$, the following formula holds
\begin{multline*} 
( [\wh \Ao_0 - z ]^{-1} f | g )_\Xf - ( [\wh \Ao - z ]^{-1} f | g )_\Xf 
 \\
 = \< [ E_0 W^{-1} + E_1  W^\cross M (z) ]^{-1}E_1 W^\cross \Ga_1 [\wh \Ao_0 - z   ]^{-1} f \mid   \Ga_1  [\wh \Ao_0 - \overline{z}  ]^{-1} g \>_\Hs .
\end{multline*}

\end{prop}

\begin{proof}
With $E_0$ and $E_1$ defined as in \eqref{e:E0E1abs}, condition \eqref{e:tK+IGa} takes the form 
$E_0 W^{-1} \Ga_0 + E_1 W^\cross \Ga_1 f =0$ as in \eqref{e:B0B1}. The operator $\wh \Ao = \Ao_\Th$  corresponds to a linear relation $\Th \subseteq \Hs_{-,+} \oplus \Hs_{+,-}$ that is connected by \eqref{e:ThW} with the linear relation $\Th_W \subseteq \Hs^2 $ of the form 
%
$\Th_W = \left\{ \{h_0, h_1\} \in \Hs^2 : E_0 h_0 + E_1 h_1 =0 \right\}$.
In order to apply \eqref{e:B0B1rho} and \eqref{e:KF}, one has to check only the assumption  
 $0 \in \rho (E_0 E_0^* + E_1 E_1^*)$ in \eqref{e:B0B1}. Since $
 E_0 E_0^* + E_1 E_1^* =  2 \wt K \wt K^* + 2 I \ge 2 I $,  the assumption  $0 \in \rho (E_0 E_0^* + E_1 E_1^*)$ is fulfilled.
Thus, \eqref{e:B0B1rho} and \eqref{e:KF} hold true. 

Note that $0 \in \rho (E_0 +E_1 M_W (z))$ in \eqref{e:B0B1rho} is equivalent to $E_0 +E_1 M_W (z) \in \Hom (\Hs) $, which we transform into  statement (i) using Remark  \ref{r:HsAdjoint} and 
the formula $M_W (z  )=W^\cross M (z) W $. Similarly, \eqref{e:KF} implies statement (ii).
\end{proof}

\begin{proof}[Proof of Theorem \ref{t:aKrF}]
In the settings of Section \ref{s:RBC}, $V \in \Hom (L^2 (\pa \D), H^{1/2} (\pa \D))$.
We apply Proposition \ref{p:KrFabst} to the m-boundary tuple 
\[
\Tc_*  =(H^{-1/2} (\pa \D), L^2 (\pa \D) , H^{1/2} (\pa \D),  \wh \gan, (- \ii)\wh \ga_0  ) =  
(\Hs_-,\Hs,\Hs_+,\Ga_0,\Ga_1) 
\] 
and to the linear homeomorphism $W = (V^\cross)^{-1} \in \Hom (L^2 (\pa \D), H^{-1/2} (\pa \D))$.
The boundary condition \eqref{e:mdisBC} associated with the m-dissipative acoustic operator $\wh \A$ 
takes the form 
\[
\ii (K+I_{L^2 (\pa \D)}) W^\cross \Ga_1 \Psi  + (K-I_{L^2 (\pa \D)}) W^{-1} \Ga_0 \Psi = 0 
\]
and corresponds to the abstract boundary condition 
\eqref{e:tK+IGa} with $\wt K = - K$.  Applying Proposition \ref{p:KrFabst},
we complete the proof of Theorem \ref{t:aKrF}.
\end{proof}

\section{Random acoustic operators and the proof of Theorem \ref{t:randM-dis}\label{s:RAOp}}
\label{s:RandomProof}

\subsection{Random bounded and random-m-dissipative operators}
\label{s:ROprelim}

Here we collect some facts about various types of random operators in a separable Hilbert space $\Xf$.  
The definition of random bounded operator in $\Xf$ (see Section \ref{s:RBC}) implies that
\begin{equation}  \label{e:RT*}
\text{$T$ is a random bounded operator if and only if $T^*$ is so.}
\end{equation}
Similarly, if $T$ and $S$ are random bounded operators, then $T+S$ is a random bounded operator.

\begin{lem}[e.g., \cite{S84}] \label{l:prodRB}
A product of two random bounded operators is also random bounded.
\end{lem}

\begin{lem} \label{l:CompR}
If $T$ is a random bounded operator, $\Om_{\Sf_\infty}:= \{\om \in \Om : T \in \Sf_\infty \} $ is an event.
\end{lem}
\begin{proof}
Let $\{f_j\}_{j \in \NN}$ be a dense subset of $\Xf$.  The countable set $\Ff$ of all finite linear combinations of the form $\sum c_{j,k} f_j \otimes f_k$ with coefficients $c_{j,k} \in \QQ$  is a dense subset in the set of all finite-rank operators. So, $\Ff$ is dense in the ideal  $\Sf_\infty$ of compact operators. Since $\de:= \inf_{Q \in \Ff} \| T - Q\|$ is a random variable, one see that $\Om_{\Sf_\infty} = \{ \om \in \Om : \de = 0\} \in \Fc$, i.e., that $\Om_{\Sf_\infty}$ is an event.
\end{proof}

\begin{prop}[Hanš theorem, e.g., \cite{BR72}] \label{p:HTh}
Let $T: \om \mapsto T_\om$ be a random bounded operator such that $T_\om \in \Hom (\Xf)$  with probability 1. Then $T^{-1}$ is a random bounded  operator. 
\end{prop}

The Hanš result establishes the randomness of $T^{-1}$ in a more general situation (see, e.g.,  \cite[Theorem 2.15]{BR72}). We need only a simple case of a separable Hilbert space $\Xf$.

\begin{rem} \label{r:RandSADis}  
Recall that random-selfadjoint and random-m-dissipative operators are defined in Section \ref{s:RBC}.
Let $T: \om \mapsto T_\om$ be an operator-valued function such that $T_\om : \dom T_\om \subseteq \Xf \to \Xf$ is a selfadjoint operator in $\Xf$ with probability 1. 
Then $T$ is random-selfadjoint if and only if $T$ is random-m-dissipative.
This follows from $(T-z)^{-1} = ((T-\overline{z})^{-1})^*$ and  \eqref{e:RT*}. 
\end{rem}

\subsection{Proof of Theorem \ref{t:randM-dis}}
\label{s:RandProof}

We prove statement (i) of Theorem \ref{t:randM-dis}. Then statement (ii) follows immediately from the last 
statement of Theorem \ref{t:AM-dis}  and Remark \ref{r:RandSADis}.

Let $V$ be a certain fixed deterministic linear homeomorphism from $L^2 (\pa \D)$ to $H^{1/2} (\pa \D)$. Given a certain random contraction $K$ in $L^2 (\pa \D)$, we define using Definition \ref{d:bc} 
the randomized acoustic operator $\wh \A $ associated with the boundary condition \eqref{e:mdisBCR}.
By Theorem \ref{t:AM-dis}, $\wh A$ is a.s. m-dissipative.
We have to prove that $\wh \A$ satisfies Definition \ref{d:rmDis} (of random-m-dissipative operator), i.e., to prove that 
$(\wh A - z)^{-1}$ is a random bounded operator for $z \in \CC_+$.

\emph{Step 1.} Theorem \ref{t:aKrF} implies that, in order to prove Theorem \ref{t:randM-dis}, it is enough to prove for every $z \in \CC_+$
  \begin{gather} 
 \text{that $E_1 V^{-1} M_\NtD (z) (V^\cross)^{-1}+ E_0 \in  \Hom ( L^2 (\pa \D))$ with probability 1}, \label{e:Inv} \\
\text{and that $[E_1 V^{-1} M_\NtD (z) (V^\cross)^{-1}+ E_0 ]^{-1} E_1$
 is a random bounded operator,} \label{e:RBproof}
\end{gather}
where $E_0$ and $E_1$ are random bounded operators in $L^2 (\pa \D)$ defined by $ E_0 := I - K$ and $ E_1 := -\ii (I+K) $. 
Indeed, the selfadjoint acoustic operator $\A^\Nr$ associated with the Neumann boundary condition $\gan (\vbf) =0$ is deterministic. Rewriting 
 \eqref{e:KFDtN} with the use of the operator $(V^\cross)^{-1} = (V^{-1})^\cross \in \Hom (L^2 (\pa \D),H^{-1/2} (\pa \D))$, one sees that it is enough to prove that  
\begin{gather}
\left( [E_1 V^{-1} M_{\NtD} (z) (V^\cross)^{-1}+ E_0 ]^{-1} E_1 V^{-1} 
 \ga_\Dr (z) \Psi  \bigm\vert  V^{-1}  \ga_\Dr (\overline{z}) \Phi \right)_{L^2 (\pa \D)} \label{e:RightPart}
 \end{gather}
is a random variable for any deterministic $\Psi$ and $\Phi$.
Note that, by Theorem \ref{t:NtD}, the NtD map $M_{\NtD} (z)$ is a linear homeomorphism from 
$H^{-1/2} (\pa \D)$ to $ H^{1/2} (\pa \D)$ for all $z \not \in \RR$, and so,
$V^{-1} M_{\NtD} (z) (V^\cross)^{-1} \in \Hom ( L^2 (\pa \D))$.
Hence, \eqref{e:RightPart} is a random variable whenever the conditions \eqref{e:Inv}-\eqref{e:RBproof} are satisfied.

\emph{Step 2.} Let us fix an arbitrary $z \in \CC_+$ and verify the conditions \eqref{e:Inv}-\eqref{e:RBproof}.
Since $\wh \A$ is a.s. an m-dissipative operator, we have $\CC_+ \subseteq \rho (\wh \A )$ with probability 1. Since $\CC \setminus \RR \subseteq \rho (\A^\Nr)$, it follows from Theorem \ref{t:aKrF} (i)  that 
$ E_1 V^{-1} M_{\NtD} (z) + E_0 V^\cross  \in \ \Hom (H^{-1/2} (\pa \D), L^2 (\pa \D))$ for a.a. $\om \in \Om$. Multiplying by $(V^\cross)^{-1}$
 from the right side, one gets \eqref{e:Inv}. That is, 
 $(E_1 V^{-1} M_{\NtD} (z) (V^\cross)^{-1}+ E_0 )^{-1} $ exists and is a bounded operator  with probability 1. 
Since $E_1 V^{-1} M_{\NtD} (z) (V^\cross)^{-1}+ E_0 $ is a random bounded operator,
the Hanš theorem (Proposition \ref{p:HTh}) implies that $(E_1 V^{-1} M_{\NtD} (z) (V^\cross)^{-1}+ E_0 )^{-1} $ is a random bounded operator. Its product with the random bounded operator $E_1$ is also a random bounded operator due to Lemma \ref{l:prodRB}. This implies  \eqref{e:RBproof} and completes the proof of Theorem \ref{t:randM-dis}.


\section{Reducing subspaces and discrete spectra}
\label{s:DiscSp}

As before, we assume \eqref{e:BT1}--\eqref{e:BT2} and fix certain linear homeomorphism  $W \in \Hom (\Hs, \Hs_{-,+})$.

 A closed subspace $\Xf_1$ of a Hilbert space $\Xf$ is called an \emph{invariant subspace} of an operator $T:\dom T \subseteq \Xf \to \Xf$ if $Tf \in \Xf_1$ for every $f \in \dom (T) \cap \Xf_1$. In this case, the operator $T$ restricted to 
$\dom  (T |_{\Xf_1}) := \dom (T) \cap \Xf_1 $ generates in $\Xf_1$ an operator $T |_{\Xf_1} : \dom (T |_{\Xf_1})  \subseteq \Xf_1 \to \Xf_1$, which is called the \emph{part of $T$ in $\Xf_1$} \cite{AG}.

\subsection{Reducing subspaces, boundary tuples, and intermediate extensions}
\label{s:SAPart}

For a closed subspace $\Xf_1$ of  a Hilbert space $\Xf$, we denote by $\Po_{\Xf_1}$ the orthogonal projection on $\Xf_1$.
Assume that $\Xf_1$ and $\Xf_2$ are invariant subspaces of $T$ such that the orthogonal decomposition 
$\Xf = \Xf_1 \oplus \Xf_2$ takes place and $\Po_{\Xf_j} f \in \dom T$ for every $f \in \dom T$ and $j=1,2$. Then $\Xf_1$ and $\Xf_2$ are called \emph{reducing subspaces} of $T$, and one says that the decomposition $\Xf = \Xf_1 \oplus \Xf_2$ \emph{reduces} $T$ to the orthogonal sum of its parts $T = T  |_{\Xf_1} \oplus T  |_{\Xf_2}$ \cite{AG}.

\begin{lem}[e.g., \cite{AG,SFBK10}]\label{l:ReducingEig}
Assume that $T$ is an m-dissipative operator or a closed densely defined symmetric operator in $\Xf$.
Assume that $\la$ is a real eigenvalue of $T$. Then the eigenspace  
$\ker (T- \la I) := \{f \in \dom T : Tf = \la f\}$ is a reducing subspace of $\Xf$.
\end{lem}

This lemma follows from \cite[Theorem 46.5]{AG} and \cite[Section IV.4]{SFBK10}.

\begin{prop} \label{p:reductionA}
Let $\la_0 \in \RR$ be an eigenvalue of $\Ao$. Then:
\item[(i)] The subspace $\Xf_0 := \ker (\Ao - \la_0 I)$ and its orthogonal complement  $\Xf_\perp := \Xf \ominus \Xf_0$ reduce the operators $\Ao$ and $\Ao^*$ to $\Ao = \la_0 I_{\Xf_0} \oplus \Ao|_{\Xf_\perp}$ and 
$\Ao^* = \la_0 I_{\Xf_0} \oplus \Ao^* |_{\Xf_\perp}$
\item[(ii)] The part $\Bo := \Ao|_{\Xf_\perp}$ of $\Ao$ is a closed symmetric operator such that its domain $\dom \Bo $ is dense in   $\Xf_\perp$ and $\Bo^* = \Ao^* |_{\Xf_\perp}$.
\item[(iii)] Let $\Ga'_j$, $j=0,1$,  be the restrictions of the map $\Ga_j$ to $\dom \Bo^*$.
Then $(\Hs_{-,+},\Hs, \Hs_{+,-}, \Ga'_0,\Ga'_1)$ is an m-boundary tuple for $\Bo^*$.
\item[(iv)] Let $\Th$ be a linear relation from $\Hs_{-,+}$ to $\Hs_{+,-}$ and let $\Ao_\Th$ and $\Bo_\Th$ be the associated intermediate extension of $\Ao$ and $\Bo$, respectively (see \eqref{e:ATh}). 
Then $\Xf = \Xf_0 \oplus \Xf_\perp$ reduces $\Ao_\Th$ to $\Ao_\Th = \la_0 I_{\Xf_0} \oplus \Bo_\Th$.
\item[(v)] An intermediate extension $\Bo_\Th$ of $\Bo$ is dissipative, m-dissipative, selfadjoint, or symmetric operator if and only $\Ao_\Th$ is so. 
\end{prop}
\begin{proof} 
Statement (i) follows from Lemma \ref{l:ReducingEig}  and,
 in the combination with the definition of adjoint operator, implies statement (ii).

Statement (iii) follows from \eqref{e:domA} and the formula 
\begin{gather} \label{e:Ga=0}
\dom (\Ao^*|_{\Xf_0}) = \Xf_0 = \dom (\Ao|_{\Xf_0}) \subset \dom \Ao \subset \ker \Ga_j , \qquad j =0,1.
\end{gather}
Indeed, \eqref{e:Ga=0} implies the surjectivity 
of $\Ga': f \to \{\Ga'_0,\Ga'_1\}$ as a map from $\dom (\Bo^*)$ to $\Hs_{-,+} \oplus \Hs_{+,-}$.
Therefore the properties of the m-boundary tuple  $(\Hs_{-,+},\Hs, \Hs_{+,-}, \Ga_0,\Ga_1)$ for $\Ao^*$
imply that $(\Hs_{-,+},\Hs, \Hs_{+,-}, \Ga'_0,\Ga'_1)$ is an m-boundary tuple for $\Bo^*$.

Combining \eqref{e:Ga=0} with statements (i)-(iii) and  \eqref{e:ATh}, one gets  
statement (iv). Combining statement (iii)-(iv) with \eqref{e:ATheqv}, we obtain statement (v).
\end{proof}

\begin{rem}\label{r:BCforAB}
In the settings of Proposition \ref{p:reductionA}, the statements on the selfadjointness and the m-dissipativity of $\Ao_\Th$ and $\Bo_\Th$ can be reformulated with the use of \eqref{e:ATheqvMdis} and Proposition \ref{p:absM-dis} 
in the following way. For every contraction $\wt K$ in $\Hs$, `abstract boundary conditions' 
 \begin{align} 
 \notag
& (\wt K+I_\Hs) W^{-1} \Ga_0 f + \ii (\wt K-I_\Hs) W^\cross  \Ga_1 f = 0 , \quad f \in \dom (\Ao^*),\\
\text{and } & (\wt K+I_\Hs) W^{-1} \Ga'_0 f + \ii (\wt K-I_\Hs) W^\cross  \Ga'_1 f = 0 , \quad f \in \dom (B^*), \label{e:K+IGa3}
\end{align}
define  the m-dissipative restrictions $\wh \Ao$ and $\wh \Bo$  of $\Ao^*$ and $\Bo^*$, respectively, in such a way that $\wh \Bo = \wh A |_{\Xf_\perp}$. The rest of statements of Propositions \ref{p:absM-dis} and \ref{p:R-RinS}
are also applicable to extensions of $\Bo$ since  $(\Hs_{-,+},\Hs, \Hs_{+,-}, \Ga'_0,\Ga'_1)$ is an m-boundary tuple for $\Bo^*$.
\end{rem}

\subsection{The proofs of Theorems \ref{t:AM-disComp} and \ref{t:BM-dis}}
\label{s:ReducingA}

The plan of this subsection is the following. First, we prove Theorem \ref{t:AM-disCompDet}, which is a deterministic analogue of Theorem \ref{t:AM-disComp}. It immediately implies its stochastic version. Then, Theorem \ref{t:BM-dis} follows from Proposition \ref{p:reductionA} and Remark \ref{r:BCforAB} combined with Lemma \ref{l:U}.

The main step is the application of the results of Section \ref{s:SAPart} to our main example, the symmetric acoustic operator $\A$ of Section \ref{s:DisBC}. Lemma \ref{l:ReducingEig} and Proposition \ref{p:reductionA} imply that the orthogonal decomposition 
\eqref{e:LabDec}
 reduces $\A$ to $\A = 0 \oplus \Bo$ and reduces $\A^*$ to $\A^* = 0 \oplus \Bo^*$, 
 \[
 \text{where  \quad $ \Bo := \A |_{\GG_{\ab,\beta}}$ \qquad and \qquad $\A^* |_{\GG_{\ab,\beta}} = (\A |_{\GG_{\ab,\beta}})^* = \Bo^*$.}
 \]

The boundary maps $\wh \gan'$
and $\wh \ga'_0$ 
are  defined as the restrictions $\wh \gan' := \wh \gan \uph_{\dom C}$ and $\wh \ga'_0 := \wh \ga'_0 \uph_{\dom C}$, i.e.,
$\wh \ga '_0 : \{\vbf,p\} \mapsto  \ga_0 (p)$ and $ \wh \gan': \{\vbf,p\} \mapsto \gan (\vbf) $ for all pairs 
$\{\vbf,p\}$ such that $\vbf \in \HH (\Div,\D) \cap \ab^{-1} \gradm H^1 (\D)$ and $p \in H^1 (\D)$.

Applying Proposition \ref{p:reductionA} and Remark \ref{r:BCforAB} to the reduction of $\Ac$ by the  decomposition $\LL^2_{\ab,\beta} (\D) = \HH_0 (\Div 0,\D) \oplus \GG_{\ab,\beta}$, we obtain the following statements. The tuple 
\[
\Tc^{\, \prime}_* :=(H^{-1/2} (\pa \D), L^2 (\pa \D) , H^{1/2} (\pa \D),  \wh \gan', (- \ii)\wh \ga'_0  ) 
\]
is an m-boundary tuple for the part $C^* = \A^* |_{\GG_{\ab,\beta}}$ of $\A^*$.
The deterministic m-dissipative extension $\wh \Ac$ of $\Ac$ defined in Section \ref{s:DisBC} via the boundary condition \eqref{e:mdisBC} admits the orthogonal decomposition 
$\wh \Ac = 0 \oplus \wh \Bo$, where  $\wh \Bo :=\wh \Ac |_{ \GG_{\ab,\beta}}$. Besides, $\wh \Bo$ is the restriction of $\Bo^*$ associated with the boundary condition 
$
(\wt K+I_\Hs) W^{-1} \wh \gan' \Psi  + (\wt K - I_\Hs) W^\cross \wh \ga'_0 \Psi  = 0,
$
where $\Psi= \{\vbf,p\}$ and $W = (V^\cross)^{-1} \in \Hom (L^2 (\pa \D), H^{-1/2} (\pa \D))$, while  $\wt K$ is connected with $K$ of \eqref{e:mdisBC} by $\wt K = - K$ (in the same way as at the end of Section \ref{s:KreinF}).

\begin{thm}\label{t:AM-disCompDet}
 Let $K$ be a deterministic contraction in $L^2 (\pa \D)$. 
Let  $\wh \A$ be an m-dissipative acoustic operator associated with  the boundary condition 
\eqref{e:mdisBC}. 
Then:
\item[(i)] The part $\wh \Bo = \wh \A|_{\GG_{\ab,\beta}}$ of  $\wh \A$  has a compact resolvent 
if and only if $K+I \in \Sf_\infty (L^2 (\pa \D))$.
\item[(ii)] If $K+I \in \Sf_\infty (L^2 (\pa \D))$, then 
$
\si (\wh \Bo ) = \si_\disc (\wh \Bo) 
$ and $\si (\wh \A) = \si_\pr (\wh \A)$.
\item[(iii)] Assume additionally that $K$ is a unitary operator.
Then $\wh \Bo = \wh \Bo^*$. Moreover, $\wh \Bo $ has a purely discrete spectrum if and only if $K+I \in \Sf_\infty (L^2 (\pa \D))$.
\end{thm}

\begin{proof}
(i) We apply Proposition \ref{p:R-RinS} to the restrictions $\wh \Bo$ and $\Bo^\Nr = \A^\Nr |_{\GG_{\ab,\beta}}$ of  $\Bo^*$, 
where $\A^\Nr$ is the acoustic operator associated with the Neumann boundary condition $\gan (\vbf) = 0$. The restriction $\wh \Bo$ corresponds in Proposition \ref{p:R-RinS} to the contraction 
$K_1 = \wt K = -K$, whereas $\Bo^\Nr$ to $K_2 = I$. Therefore  
$(\wh \Bo - \la)^{-1} - (\Bo^\Nr  - \la)^{-1} \in \Sf_\infty (\GG_{\ab,\beta})$
for all (or at least for one) $\la \in \rho (\wh A) \cap \rho (\A^\Nr)$
if and only if $I - \wt K = I + K \in \Sf_\infty (L^2 (\pa \D))$.

Since the selfadjoint operator $\Bo^\Nr = \A^\Nr |  \GG_{\ab,\beta}$ has a compact resolvent 
\cite{L13}, we see that $(\wh \Bo - \la)^{-1} \in \Sf_\infty (\GG_{\ab,\beta})$ for $\la \in \rho (\wh \Bo)$ if and only if $I + K \in \Sf_\infty (L^2 (\pa \D))$.

(ii) Let $K+I \in \Sf_\infty (L^2 (\pa \D))$. It follows from (i) and Remark \ref{r:DiscDet}  that $ \si (\wh \Bo ) = \si_\disc (\wh \Bo) = \si_p (\Bo)$. Since $\wh \A = 0 \oplus \wh \Bo$ and $\ker \wh \A \supseteq \ker \A \neq \{0\}$, we get  
$\si (\wh \A) = \si_p (\wh \A) = \si (\wh \Bo) \cup \{0\}$.

(iii) Statement (ii), Remark \ref{r:BCforAB}, and Proposition \ref{p:absM-dis} imply statement (iii).
\end{proof}

\begin{proof}[Proof of Theorem \ref{t:AM-disComp}]
The application of Theorem \ref{t:AM-disCompDet} pointwise in $\om$ (on a suitable  event $\Om_1 \subset \Om$ of probability 1)  to a random contraction $K:\om \mapsto K_\om$ proves Theorem \ref{t:AM-disComp}.
\end{proof}

\begin{proof}[Proof of Theorem \ref{t:BM-dis}]
The unitary equivalence of Lemma \ref{l:U} establishes a bijective correspondence between m-dissipative extensions $\wh \Bo = \wh \A|_{\GG_{\ab,\beta}}$ of $\Bo$ and m-dissipative extensions of the operator $\Bc$ of Section \ref{s:El}. Combining this bijective correspondence with Theorems \ref{t:AM-dis}, \ref{t:randM-dis}, \ref{t:AM-disComp}, \ref{t:AM-disCompDet} and Remark \ref{r:BCforAB}, one obtains 
all the statements of Theorem \ref{t:BM-dis}.
\end{proof}

\section{Discussion and additional remarks}
\label{s:dis}

\subsection{Random essential spectra and Weyl's law on rough boundaries}

Section \ref{s:RImp} establishes a connection between the asymptotics of the discrete spectrum $\{\mu_j\}_{j \in \NN}$ of the Laplace-Beltrami operator $\De^{\pa \D}$ on the Lipschitz boundary $\pa \D$ and  the existence of essential spectra for acoustic operators $\wh \Bc^{\diag(\zeta)}$,
which are associated in the domain  $\D$ with stochastic generalized impedance boundary conditions $\Zc^{\diag(\zeta)} \ga_0 (p) = - \gan (\ab^{-1} \nabla u)$. 

In particular, the stochastic results of Section \ref{s:RImp} imply the following deterministic results in the  case of a sequence $\zeta=\{\zeta_j\}_{j \in \NN}$ of deterministic complex  numbers.  The deterministic 2nd order acoustic operator  $\wh \Bc^{\diag(\zeta)}$ (see Remarks \ref{r:Qimp}-\ref{r:Qimp2} and \eqref{e:Zfin}) is an m-dissipative operator in $\LL^2_{\ab,\beta} (\D) $ if and only if $\re \zeta_j  \ge 0 $ for all $j \in \NN$. In this case, the resolvent of $\wh \Bc^{\diag(\zeta)}$ is compact exactly when $\lim_{j\to \infty}  \frac{\zeta_j}{\sqrt{\mu_j}}  = 0$.  The operator $\wh \Bc^{\diag(\zeta)}$ is selfadjoint  if and only if  $\re \zeta_j = 0$ for all $j \in \NN$. In the selfadjoint case,  the spectrum of $\wh \Bc^{\diag(\zeta)}$ is purely discrete if and only if  $\lim_{j\to \infty} \frac{\zeta_j}{\sqrt{\mu_j}} = 0$.

The stochastic results of Theorem \ref{t:iid} and Corollary \ref{c:Weyl} connect the case of i.i.d. random diagonal entries $\zeta_j$ with the Weyl law of  the counting function $\Nc_{\De^{\pa \D}} (\la) =  \sum_{\mu_j \le \la} 1 $ for eigenvalues of 
$\De^{\pa \D}$. 

In the case of a sufficient $C^k$-regularity of the boundaries $\pa \D$ (or more generally, of Riemannian manifolds),
the asymptotic of $\Nc_{\De^{\pa \D}} $ was studied intensively by the methods of microlocal analysis, see \cite{Z98,Z99,I00,I19,GKLP22} and the references therein. These studies of Weyl's law  were concentrated on the sharp estimate of the 2nd asymptotic term and on the optimal  regularity of manifolds (or coefficients) that ensure these sharp estimates.

In the case of $k <2$,  the results of \cite{Z98,Z99} on the intermediate remainder estimates apply to the case of a $C^{1,\vep}$-boundary $\pa \D$ with $\vep>0$. The asymptotics of these results can be used in Corollary \ref{c:Weyl}  since the results of \cite{Z98,Z99} are much more accurate (up to the understanding of the author of present paper) than the rough estimate \eqref{e:asymp} that we need. 

The methods of \cite{Z98,Z99,I00} may be difficult to apply to general Lipschitz boundaries.   However, we  need so rough estimate  \eqref{e:asymp} on the 1st term in Weyl's asymptotics that there is a hope that \eqref{e:asymp}  is valid for every Lipschitz boundary and can be obtained by simpler methods, e.g., by the Dirichlet-Neumann bracketing. 

\subsection{Boundary conditions for elliptic equations in the divergence form} \label{s:ElBC}

The adaptation of the deterministic and random m-dissipative boundary conditions introduced in this paper to the case of the 2nd order elliptic equation  in the divergence form 
\begin{gather}\label{e:EqEl}
\text{$- \Div ( \al^{-1} \gradm p)  =  \la^2 \beta p$  ,  \qquad $p \in H^1 (\D) $,}
\end{gather}
leads formally to the $\la$-dependent boundary condition 
\begin{gather} \label{e:laBC}
(K+I) (I+\De^\paD)^{1/4} \ga_0 (p)  + \frac1{\ii \la} (K-I) (I+\De^\paD)^{-1/4} \gan (\ab^{-1} \nabla p) = 0 .
\end{gather}
This reduction is not completely equivalent. The issue is not only in the explicit exclusion of the spectral parameter $\la_0 = 0$,
but also in  the treatment of the points of spectra that are not eigenvalues, which remains ambiguous for \eqref{e:EqEl}-\eqref{e:laBC}. Sections \ref{s:PartComp} shows that such points may exist. However, since \eqref{e:EqEl}-\eqref{e:laBC} does not define an operator directly,
one has to go back to the operators defined by acoustic systems in order to define the associated essential or continuous spectra.
This issue may be partially resolved by an adaptation of the abstract Weyl function results of \cite[Proposition 1.6]{DM95}  to  DtN-maps and m-boundary tuples.

\subsection{Comparison of various types of Dirichlet-to-Neumann maps} \label{s:DtNdis}

In connection with a variety of applications, the DtN-maps depending on the spectral parameter have attracted considerable attention \cite{KG89,F91,P12,CK13,AGW14,BtE15,DM17,I19,BHdS20,AtE20,BtE21,FJL21,DHM22,GKLP22}. Many of the studies were concentrated on the cases of time-harmonic Schrödinger equations and on elliptic equations  in the divergence form $\Div \ab^{-1} \gradm u + (\ka -q)u = 0$, where the spectral parameter $\ka$ corresponds to $\la^2$ of \eqref{e:EqEl}.

Let us connect the results of Sections \ref{s:NtDKrRes} and \ref{s:DtN} on the acoustic DtN maps $M_\DtN$ with the DtN maps of the papers \cite{AtE20,BtE21}, which we denote by $m_\DtN $ and which seems to be the closest to our settings since \cite{AtE20,BtE21} include as a particular case divergence-type equations $\Div \ab^{-1} \gradm u + \ka u = 0$ in a Lipschitz domain $\D$ with a wide class of $L^\infty$-matrix-valued coefficient $\ab^{-1} (\cdot)$. 

In Remark \ref{r:DtN}, we define the DtN map $M_\DtN (\la)$ associated with the equation $\af_{\ab,\beta} \Psi = \la \Psi$ for $\la \in \rho (\A^\Dr) \cap \rho (\A^\Nr) \supset \CC \setminus \RR$.
For such spectral parameters $\la$, it is easy to see that $M_\DtN (\la)$ coincides with the $\Ga_0$-to-$\Ga_1$ map corresponding to the acoustic operator $\A^*$ and the  associated m-boundary tuple $(H^{1/2} (\pa \D), L^2 (\pa \D) , H^{-1/2} (\pa \D),  \wh \ga_0, \ (-\ii) \wh \gan )$. Proposition \ref{p:WeylF} implies that 
$M_\DtN (\cdot)$  admits an extension to a holomorphic  $\Lc (H^{1/2} (\paD),H^{-1/2} (\paD))$-valued function in $\rho (\A^\Dr)$. We keep the notation $M_\DtN (\cdot)$ for this extension and notice that it possesses the properties (i)-(iii) of Proposition \ref{p:WeylF}.

In particular, by Proposition \ref{p:WeylF}, the DtN map $M_\DtN (\la)$  is completely determined for $\la \in \rho (\A^\Dr)$ by the requirement 
that $M_\DtN (\la) \ga_0 (p) = - \ii \gan (\vbf)$ for all solutions 
to the system 
\begin{gather} \label{e:ASys2}
\text{$ \ab^{-1} \gradm  p = \ii \la  \vbf$, \qquad $ \beta^{-1} \Div \vbf = \ii \la  p$, \qquad $\{\vbf,p\} \in \HH (\Div,\D) \times H^1 (\D) $.}
\end{gather}
Since $\la  \neq 0$ for $\la \in \rho (\A^\Dr)$ and since $\vbf = \frac{1}{\ii \la} \ab^{-1} \gradm p$, we see that for $\la \in \rho (\A^\Dr)$ this system is equivalent to equation \eqref{e:EqEl}.

We define the divergence-type operator $L^{\Dr} = - \beta^{-1} \Div  \al^{-1} \gradm_0$ as an operator in the Hilbert space $L^2_\beta (\D)$ with the domain 
$\dom L^\Dr = \{ p \in H^1_0 (\D) : \al^{-1} \nabla p \in \HH (\Div,\D)\}$.
The results of \cite[Sections 3, 4, and 7]{L13} imply that $L^\Dr$ is selfadjoint with a purely discrete spectrum 
$\si (L^\Dr) = \si_\disc (L^\Dr) \subset (0,+\infty)$  and, in turn, that 
$\si (L^\Dr)  = \{ \la^2 : \la \in \si (\A^\Dr)  \} \setminus \{0\}  $
(note that  $\si (\A^\Dr) $ is real and symmetric w.r.t. $\la_0=0$).

Modifying somewhat the settings of \cite{AtE20,BtE21}, one defines for 
$\ka \in \rho (L^\Dr)$  the DtN map $m_\DtN (\ka)$ associated with equation \eqref{e:EqEl} by the equality 
\[
\text{$m_\DtN (\ka) \ga_0 (p) = \gan (\al^{-1} \nabla p)$ }
\] 
that is supposed  to be valid for all $p$ satisfying \eqref{e:EqEl} with $\la^2 = \ka$. 
One sees from the preceding arguments, that the connection between the two types of DtN maps is given by the formula 
\begin{equation} \label{e:M=m}
\text{$ M_\DtN (\la) = - \frac{1}{\la} m_{\DtN} (\la^2) $ \quad for all $\la \in \rho (\A^\Dr)$,}
\end{equation}
and that $\lim\limits_{\la \to 0} ( \la M_\DtN (\la)) = - m_{\DtN} (0)$. 


\appendix

\section{Appendix: Generalized rigged Hilbert spaces}
\label{ap:A}

A nonnegative  operator $T$ is called positive 
 if it has the trivial kernel $\ker T = \{0\}$.
Let an abstract complex Hilbert space $\Hs$ play the role of a pivot space. 
In the context of trace spaces associated with various types of wave equations, the following general type of duality w.r.t. a pivot Hilbert space  $\Hs$ arises naturally, see \cite[Vol.2, Appendix to IX.4, Ex.~3]{RS} and \cite{S21,EK22}.
 
\begin{defn}[{generalized rigged Hilbert space \cite[Vol.2]{RS}}] \label{d:GRHSp}
Assume that Hilbert spaces  \linebreak $(\Hs_{-,+}, \|\cdot\|_{\Hs_{-,+}})$ and $(\Hs_{+,-},\|\cdot\|_{\Hs_{+,-}})$ satisfy the following conditions:
\begin{itemize}
\item[(G1)] The intersections $\Hs \cap \Hs_{\mp,\pm} $ are dense in the pivot space $\Hs$ w.r.t. its norm $\|\cdot\|_\Hs$, and $\Hs \cap \Hs_{\mp,\pm} $ is dense in $\Hs_{\mp,\pm}$ w.r.t.  $\|\cdot\|_{\Hs_{\mp,\pm}}$;
\item[(G2)] there exist a pair of operators, $\So_{-,+}$ and $\So_{+,-}$, such that 
$\So_{\mp,\pm}$ are positive selfadjoint operators in $\Hs$, $\dom \So_{\mp,\pm} =\Hs \cap \Hs_{\mp,\pm} $, and 
  $\| h \|_{\Hs_{\mp,\pm}} = \| \So_{\mp,\pm} h\|_{\Hs}$ for all $h \in \Hs \cap \Hs_{\mp,\pm} $;
\item[(G3)] $\So_{-,+} = \So_{+,-}^{-1}$ (this implies that $\Hs_{-,+} $ and $\Hs_{+,-}$ are dual to each other w.r.t. $\Hs$).
\end{itemize}
Then the triple $(\Hs_{-,+},\Hs,\Hs_{+,-})$ is called a generalized rigged Hilbert space.  
\end{defn}

The mutual duality of $\Hs_{-,+} $ and $\Hs_{+,-}$ in  (G3)  is understood in the following  sense. The inner product $(\cdot,\cdot)_\Hs$ of $\Hs$ can be extended in a unique way by continuity to two sesquilinear pairings  between $\Hs_{\mp,\pm}$ and $\Hs_{\pm,\mp}$. We denote each of these two pairings by $\<\cdot,\cdot\>_\Hs$ so that $\<h_{-,+},h_{+,-}\>_\Hs = \overline{\<h_{+,-},h_{-,+}\>_\Hs}$ for $h_{\mp,\pm} \in \Hs_{\mp,\pm}$. 
These pairings produce natural identifications between $\Hs_{\pm,\mp}$ and  the adjoint spaces 
$\Hs_{\mp,\pm}^* $, see for details \cite{EK22}.


\end{document}